\documentclass[a4paper,10pt,intlimits]{amsart}

\usepackage[T1]{fontenc}
\usepackage[latin1]{inputenc}
\usepackage[english]{babel}
\usepackage{amsmath}
\usepackage{amsfonts}
\usepackage{amssymb}
\usepackage{amsthm}
\usepackage{enumerate}
\usepackage{dsfont}

\newcommand{\calB}{\mathcal{B}}
\newcommand{\calD}{\mathcal{D}}
\newcommand{\calE}{\mathcal{E}}
\newcommand{\calH}{\mathcal{H}}

\newcommand{\calM}{\mathcal{M}}

\newcommand{\CC}{\mathds{C}}
\newcommand{\NN}{\mathds{N}}
\newcommand{\RR}{\mathds{R}}
\newcommand{\ZZ}{\mathds{Z}}

\DeclareMathOperator{\supp}{supp}
\DeclareMathOperator{\capp}{cap}

\newcommand{\One}{\mathds{1}}
\newcommand{\rmd}{\mathrm{d}} % Abstand \:

\newtheorem{theorem}{Theorem}[section]
\newtheorem{corollary}[theorem]{Corollary}
\newtheorem{lemma}[theorem]{Lemma}
\newtheorem{proposition}[theorem]{Proposition}

\theoremstyle{definition}
\newtheorem{definition}[theorem]{Definition}
\newtheorem{example}[theorem]{Example}

\theoremstyle{remark}
\newtheorem{remark}[theorem]{Remark}
\def\clap#1{\hbox to 0pt {\hss#1\hss}}
\def\mathclap{\mathpalette\mathclapinternal}
\def\mathclapinternal#1#2{\clap{$\mathsurround=0pt#1{#2}$}}

\newcommand{\sint}[1]{\int_{\mathclap{#1}} \;}

\newcommand{\Hmm}[1]{\leavevmode{\marginpar{\tiny%
$\hbox to 0mm{\hspace*{-0.5mm}$\leftarrow$\hss}%
\vcenter{\vrule depth 0.1mm height 0.1mm width \the\marginparwidth}%
\hbox to 0mm{\hss$\rightarrow$\hspace*{-0.5mm}}$\\\relax\raggedright #1}}}

\begin{document}

\title[Intrinsic metrics for non-local forms and applications]{Intrinsic metrics for non-local symmetric Dirichlet forms and applications to spectral theory}

\author[R.~L.~Frank]{Rupert L. Frank$^1$}
\address{$^1$ Department of Mathematics, Princeton University, Princeton, NJ 08544, USA }
\email{ rlfrank@Math.Princeton.edu}

\author[D.~Lenz]{Daniel Lenz$^2$}
\address{$^2$ Mathematisches Institut, Friedrich-Schiller Universit\"at Jena, Ernst-Abb\'{e} Platz~2, 07743 Jena, Germany}
\email{ daniel.lenz@uni-jena.de }

\author[D.~Wingert]{Daniel Wingert$^3$}
\address{$^3$ Fakult\"at f\"ur Mathematik, TU Chemnitz, 09107 Chemnitz, Germany}
\email{daniel.wingert@s2000.tu-chemnitz.de}

\begin{abstract} We present a study of what may be called an intrinsic metric for a general regular Dirichlet form. For such forms we then prove a Rademacher type theorem. For strongly local forms we show existence of a maximal intrinsic metric (under a weak continuity condition) and for Dirichlet forms with an absolutely continuous jump kernel we characterize intrinsic metrics by bounds on certain integrals.
We then turn to applications on spectral theory and provide for (measure perturbation of) general regular Dirichlet forms an Allegretto-Piepenbrinck type theorem, which is based on a ground state transform, and a Shnol type theorem. Our setting includes Laplacian on manifolds, on graphs and $\alpha$-stable processes.
\end{abstract}

\maketitle

\section{Introduction}
Intrinsic metrics play an important role in the study of various features of Laplacians on manifolds and more generally of Laplacians arising from strongly local Dirichlet forms. In this context, they appear for example in the study of heat kernel estimates \cite{Sturm-94c,Sturm-96}, the investigation of stochastic completeness, recurrence and transience \cite{Sturm-94b} and spectral theory see e.g.\ \cite{BoutetdeMonvelS-03b,BoutetdeMonvelLS-08,LenzSV}. Thus, basic issues for intrinsic metrics can be considered to be well understood in the case of strongly local Dirichlet forms. For non-local Dirichlet forms the situation is completely different. In fact, already for the simplest examples, viz.\ graphs, there is no common concept of an intrinsic metric (see, however, \cite{Davies-93} for various ideas in this direction). This is the starting point for this paper.

Our basic aim is to propose an extension of the concept of intrinsic metric from strongly local Dirichlet forms to the general case  and to study some of its basic features. More precisely, we proceed as follows:

After presenting the basic ingredients of Dirichlet forms in Section \ref{Preliminaries}, we carry out a careful study of energy measure and of a suitable space of functions to be thought to belong locally to the domain in Section \ref{Energy}. In particular, we prove a certain continuity of the energy measures in Proposition \ref{prop:continuity} and discuss variants of the Leibniz rule.

In Section \ref{Intrinsicmetric} we then present a general concept  of intrinsic metric and study some of its properties. In particular, in Theorem \ref{thm:cool} we provide a Rademacher type theorem in a rather general context. This theorem has already proven useful in Stollmann's study of length spaces \cite{Stollmann-09}. In Section \ref{Intrinsic-metric-functions}, we show that specifying an intrinsic metric more or less amounts to specifying a set of Lipschitz continuous functions.

Combining these results we then have a look at the  strongly local case in our context in Section \ref{sec:strongly_local}. In this case it is possible to show existence of a maximal intrinsic metric (Theorem \ref{maximal-intrinsic-metric}). Existence of a maximal intrinsic metric fails in general for the non-local case, as we show by examples. In this sense, our results 'prove' that the non-local case is strictly more complicated than the local case as far as intrinsic metrics are concerned.

The situation of an absolutely continuous jump kernel is considered in Section \ref{Absolutely}. There, we can then characterize our intrinsic metrics by some integral type condition (Theorem \ref{thm:converse}). This is well in line with earlier results and ideas on e.g.\ graphs.

Dirichlet forms with finite jump size are considered in Section \ref{Finite}. In a precise sense, these turn out to be not much different from strongly local forms.

After these more geometric considerations we turn to spectral theory and present two applications of the developed theory. Both applications rely on the notion of generalized eigenfunction which in turn is defined using the local domain of definition. In their context, we actually allow furthermore for some perturbation of the original Dirichlet form by a (suitable) measure.

The first application, given in Section \ref{Ground} provides a ground state transform and then an Allegretto-Piepenbrink type result. This basically unifies the corresponding results of \cite{HK,LenzSV} for graphs and strongly local forms respectively.

The second application concerns a Shnol' type result. Such a result was recently shown in \cite{BoutetdeMonvelS-03b} for strongly local forms using cut-off functions induced by the intrinsic metric. Having the intrinsic metrics at our disposal, we can adapt the strategy of \cite{BoutetdeMonvelS-03b} to our general context.

Examples such as graphs and $\alpha$-stable processes to which our results can be applied are discussed in the last section.

\smallskip

For related material concerning heat semigroup estimates we refer to reader to \cite{Wingert-10}. After this work was finished we learned about recent work of Grigor'yan, Huang  and Masamune \cite{GHM} on stochastic completeness, which seems to have some points of contact with our considerations. 

\bigskip

\textbf{Acknowledgements.} D.L.\ would like to thank Peter Stollmann, Matthias Keller and Sebastian Haeseler  for most stimulating discussions.

\section{Preliminaries on Dirichlet forms} \label{Preliminaries}

This paper is concerned with Dirichlet forms on locally compact separable spaces. In this section we introduce the basic notation and concepts (see e.g.\ \cite{BouleauH-91, Davies-90b, FukushimaOT-94, MaR-92}.)

\medskip

Throughout let $X$ be a locally compact, separable metric space, $m$ a positive Radon measure on X with $\supp m = X$. The functions on $X$ we consider will all be real valued. Of course, complex valued functions could easily be considered as well after complexifying the corresponding Hilbert spaces and forms. By $C_c (X)$ we denote the set of continuous functions on $X$ with compact support. The space $L^2 (X):= L^2 (X,m)$ is the space of all measurable (square integrable with respect to $m$) real valued functions. The space $L^\infty (X,m)$ is the space of all essentially bounded functions.

\textbf{Notation.} We will mostly replace the argument $(u,u)$ by $(u)$ in all sorts of bilinear maps. Thus, in the context of forms  we will use the notation $\calE(u)$ instead of $\calE(u,u)$ and similarly $\mu^{*} (u)$ for $\mu^{*} (u,u)$ for measure valued maps etc.\ at the corresponding places below.

\subsection*{Dirichlet forms}

Recall that a closed non-negative form on $L^2(X,m)$ consists of a dense subspace $\calD \subset L^2(X,m)$ and a sesquilinear and non-negative map $\calE : \calD \times \calD \rightarrow \RR$ such that $\calD$ is complete with respect to the energy norm $\sqrt{\calE_1}$ defined by
\[ \calE_1 (u) =\calE(u) + \| u \|_{L^2(X,m)}^2. \]
In this case the space $\calD$ together with the inner product $\calE_1(u, v):=\calE(u, v) + (u, v)$ becomes a Hilbert space and $\sqrt{\calE_1}$ is the induced norm.

A closed form is said to be a \emph{Dirichlet form} if for any $u \in \calD$ and any normal contraction $T : \RR \rightarrow \RR$ we have also
\[ T\circ u\in \calD\mbox{ and }\calE(T\circ u)\le \calE(u). \]
Here, $T:\RR \to\RR$ is called a \emph{normal contraction} if
$T(0)=0$ and $|T(\xi)-T(\zeta)|\le |\xi -\zeta|$ for any
$\xi,\zeta\in\RR$. A Dirichlet form is called \emph{regular} if $\calD \cap
C_c(X)$ is dense both in $(\calD, \| \cdot \|_{\sqrt{\calE_1}})$ and
$(C_c(X), \| \cdot \|_{\infty })$.

In the remainder of this paper, we shall assume that $\calE$ is a regular Dirichlet form.

\smallskip

\subsection*{Capacity}

The capacity is a set function associated to a Dirichlet form. It measures the size of sets adapted to the form. It is defined as follows: For $U\subset X$, $U$ open, we define
\[ \capp(U) := \inf \{ \calE_1 (v) \mid v \in \calD, \One_U \le v \}, \qquad (\inf \emptyset = \infty), \]
For arbitrary $A \subset X$, we then set
\[ \capp(A) := \inf \{ \capp(U) \mid A \subset U \} \]
(see \cite{FukushimaOT-94}, Section 2.1). A property is said to hold \emph{quasi-everywhere}, short \emph{q.e.}, if it holds outside a set of capacity $0$. A function $f : X \to \RR$ is called \emph{quasi-continuous}, \emph{q.c.}\ for short, if, for any $\varepsilon >0$ there is an open set $U \subset X$ with $\capp(U) \le \varepsilon$ so that the restriction of $f$ to $X \setminus U$ is continuous. Every $u \in \calD$ admits a q.c.\ representative $\tilde{u} \in u$ (recall that $u \in L^2(X,m)$ is an equivalence class of functions) and two such q.c.\ representatives agree q.e. For $u \in \calD$, we will always choose such a quasicontinuous representative $\tilde{u}$.

\smallskip

\subsection*{Algebraic structure}

Here we highlight the following theorem.

\begin{theorem}[{\cite[Thm. 1.4.2.(ii)]{FukushimaOT-94}}]
\label{thm:algebraic_structure}
 Let $u, v \in \calD \cap L^\infty (X)$. Then $u v \in \calD$ and the estimates
 \[ \calE (uv) \leq \|u\|_\infty^2 \calE (v) + \|v\|_\infty^2 \calE (u) \]
 holds.
\end{theorem}

Note that the theorem implies that the vector spaces $L^\infty (X) \cap \calD$, $C_c (X) \cap \calD$ and $L^\infty_c (X) \cap \calD$ are actually algebras (i.e., closed under multiplication).

\smallskip

\subsection*{Beurling-Deny formula}

There is a fundamental representation theorem for regular Dirichlet forms, known as Beurling-Deny formula. It says that to any such form $\calE$ there exists
\begin{itemize}
 \item $k$, a (non-negative) Radon measure on $X$,
 \item $J$, a (non-negative) Radon measure on $X \times X - d$, i.e., on $X \times X$ without the diagonal $d=\{(x,x): x\in X\}$,
 \item and $\mu^{(c)}$, a positive semidefinite bilinear form on $\calD$ with values in the signed Radon measures on $X$, which is strongly local, i.e., satisfies $\mu^{(c)}(u, v) = 0$ if $u$ is constant on $\supp v$,
\end{itemize}
such that for any $u \in \calD$ with q.c.\ representative $\tilde{u}$
\[ \calE(u) = \sint{X \times X - d} (\tilde{u}(x) - \tilde{u}(y))^2 \: J(\rmd x, \rmd y) + \int_X \: \rmd \mu^{(c)}(u) + \int_X \tilde{u}(x)^2 \: k(\rmd x). \]
If $J \equiv 0$, then $\calE$ is called \emph{local}.

\smallskip

The measure $J$ gives rise to the Radon measure $\mu^{(b)}$ characterized by
\[ \int_K \rmd\mu^{(b)}(u) = \sint{K \times X - d} (\tilde{u}(x) - \tilde{u}(y))^2 \: J(\rmd x, \rmd y) \]
for $K \subset X$ compact and $u \in \calD$. The measure $\mu^{(d)}$ is defined as
\[ \mu^{(d)} = \mu^{(c)} + \mu^{(b)}. \]
We then define the measure $\mu^{(a)}$
\[ \int_K \rmd \mu^{(a)}(u) = \int_K \tilde{u}^2 \: \rmd k. \]
Finally, we define the bilinear forms $\calE^{(*)}$ for $*=\cdot, a,b,c$ by
\[ \calE^{(*)}(u, v) = \int_X \rmd \mu^{(*)}(u, v) \]
 and call $\mu^{(c)}$ the \emph{strongly local part of the energy measure}, $\mu^{(b)}$ the \emph{jump part} or the \emph{pseudo-differential part} or the \emph{non-local part of the energy measure} and $\mu^{(d)}$ the \emph{differential part of the energy measure}. With this notation, the Beurling-Deny formula reads
\[ \calE(u) =\calE^{(c)}(u) +\calE^{(b)}(u) +\calE^{(a)}(u) =\calE^{(d)}(u) +\calE^{(a)}(u) . \]

From the definitions, it is not hard to see that the measures $\mu^{(*)}(u, v)$, $*=\cdot,a, b,c,d$ satisfy the Cauchy-Schwarz inequality
\[ \left(\int |f g| \rmd \mu^{(*)} (u, v) \right)^2 \leq \int |f|^2 \rmd \mu^{(*)} (u) \int |g|^2 \rmd \mu^{(*)} (v, v) \]
for all measurable $f, g$ on $X$ and all $u, v\in \calD (\calE)$.

\medskip

Let us also note that  the strongly local part $\mu^{(c)}$ satisfies the truncation property (\cite{Mosco-94})

\[ \rmd \mu^{(c)}(u \vee v, u \vee v) = \One_{\{u > v\}} \rmd \mu^{(c)}(u) + \One_{\{u \leq v\}} \rmd \mu^{(c)}(v, v). \]
Here, we write for real valued functions $f,g$ on the same space
\[ f\wedge g:= \min\{f,g\}, \;\: f \vee g :=\max\{f,g\}. \]

\bigskip

The measures $\mu^{(*)}(u, v)$, $* = a, b, c, d$ depend continuously on $u \in (\calD,\calE_1)$, when the space of measures is equiped with the vague topology (see e.g.\ \cite{Sturm-94b}). We will need a somewhat different continuity, which may be of independent interest. It is stated in the following proposition. A proof can be found in the appendix.

\begin{proposition}\label{prop:continuity}
 Let $u_n, u \in \calD \cap L^\infty(X)$ and $v \in \calD$ be given with $\| u_n \|_\infty$ bounded and $u_n \stackrel{}{\longrightarrow} u$ weakly with respect to $ \sqrt{\calE_1}$. Then, the measures $\mu^{(*)} (u_n,v)$ converge vaguely to the measure $\mu^{(*)} (u,v)$, i.e.,
 \[ \int f \rmd \mu^{(*)}(u_n, v) \rightarrow \int f \rmd \mu^{(*)}(u, v) \]
 holds for all $f \in C_c (X)$ for $*=a,b,c,d$.
\end{proposition}

\section{Energy measure and the space $\calD_{loc}^*$} \label{Energy}

\subsection{The local form domain $\calD_{loc}^*$}\label{sec:locformdom}

Functions which locally belong to the domain of the form play a crucial role in the theory of Dirichlet forms. This holds in particular for the use of strongly local forms in the theory of partial differential equations. In that case there is a well known space $\calD_{loc}$ (defined below) with the following three properties:

\begin{itemize}
 \item Functions in $\calD_{loc}$ locally agree with functions in $\calD$.
 \item The measure valued functions $\mu^{(a)}$ and $\mu^{(c)}$ can be extended to $\calD_{loc}$.
 \item The form $\calE$ can be extended so that $\calE(u, v)$ makes sense for all $u, v \in \calD_{loc}$ such that $v$ has compact support.
\end{itemize}
Here, we introduce a space which has these three properties and agrees with the usually defined $\calD_{loc}$ in the strongly local case. The basic idea is to find the 'biggest' set of functions to which the measures $\mu^{(a)}$ to $\mu^{(c)}$ can be extended.

\medskip

In \cite{FukushimaOT-94} one can find the set of functions, which are locally in the domain of definition of $\calE$, defined as
\[ \{u \in L^2_{loc} \mid \;\forall G \subset X \text{ open, relatively compact } \;\; \exists {v} \in \calD \text{ with } u = {v} \text{ on } G \}. \]
We will denote this set as $ \calD_{loc}$.

 Some of its properties are gathered in the following proposition. These properties are of course well known. For the convenience of the readers we provide the short proofs.

\begin{proposition} Let $\calE$ be a regular Dirichlet form on $(X, m)$. Then, the following holds.
 \begin{itemize}
 \item[(a)] $\calD \subset \calD_{loc}$.
 \item[(b)] For all $u \in \calD_{loc}$ there exists a quasi-continuous version $\tilde{u}$.
 \item[(c)] $1 \in \calD_{loc}$.
 \end{itemize}
\end{proposition}
\begin{proof}
 (a) is clear.

 (b) Let $u \in \calD_{loc}$. We choose $G_n$ open, relatively compact with $\bigcup_n G_n = X$ and $G_n \subset G_{n+1}$. Then there exists $u_n \in \calD$ quasi-continuous with $u_n = u$ m-a.e.\ on $G_n$. According to \cite[Theorem 2.1.2]{FukushimaOT-94} there exists a m-regular nest $F_k$ (i.e., $\capp(X \setminus F_k) \rightarrow 0$ and $m(F_k \cap U(x)) > 0 \;\forall x \in F_k$ and $U(x)$ a neighbourhood of $x$) with $u_n \big|_{F_k}$ continuous for all $n$ and $k$. Let now be $x \in F_k \cap G_n$ for some $n$. \\
 Because $m(U(x) \cap F_k) > 0$, there is a point $x_U \subset U(x) \cap F_k$ with $u_m(x_U) = u(x_U) = u_n(x_U)$ for all $m \geq n$. As $x_U \rightarrow x$ for appropriately  chosen $U(x)$ and on acount of the continuity it follows that $u_m(x) = u_n(x)$ for all $m \geq n$. Hence $u_m(x) = u_n(x)$ for all $m \geq n$ and $x \in F_k \cap G_n$. \\
 We now define $\tilde{u}(x) := u_n(x)$ for $x \in F_k \cap G_n$. $\tilde{u}$ is well defined on every $F_k$ and continuous. Therefore $\tilde{u}$ is quasi-continuous, and as $m(X \setminus F_k) \leq \capp(X \setminus F_k) \rightarrow 0$ we know that $\tilde{u}$ is a quasi-continuous version of $u$.

 (c) This is a direct consequence of \cite[Lemma 1.4.2.(ii)]{FukushimaOT-94} (for all $G$ open, relatively compact there exists a $u \in \calD \cap C_c(X)$ with $u(G) \equiv 1$).
\end{proof}

For $u\in\calD_{loc}$ we will always choose a quasi-continuous representative $\tilde u$ as in part (b) of the proposition, and we shall often simply write $u$ for this representative.

Note that $\mu^{(c)}$ can be extended to $\calD_{loc}$ because of its local property. To be precise: $\mu^{(c)}(u) = \mu^{(c)}(v)$ on $G \subset X$ open, relatively compact, whenever $u = v$ on $G$, see \cite[Remarks to Thm. 3.2.2.]{FukushimaOT-94}. This $\mu^{(c)}$ is again called the strongly local part of the energy measure.
In order to be able to extend $\mu^{(b)}$ as well, we will have to restrict our attention to a certain subset of $ \calD_{loc}$ which we define next.

\begin{definition}
 The space $\calD_{loc}^*$ of \emph{functions locally in domain} is defined to be the set of all functions $u \in \calD_{loc}$ with the property that
 \[ \sint{K \times X - d} (\tilde{u}(x) - \tilde{u}(y))^2 \: J(\rmd x, \rmd y) < \infty \]
 for all compact $K\subset X$.
\end{definition}

We can extend $\mu^{(b)}$ to the space $\calD_{loc}^*$. To do so, we define for $E\subset X$ measurable and $u\in \calD_{loc}^*$
\[ \mu^{(b)}(u)(E) := \mu^{(b)} (u):= \sint{E \times X - d} (\tilde{u}(x) - \tilde{u}(y))^2 \: J(\rmd x, \rmd y). \]

\begin{proposition}
 For $u \in \calD_{loc}^*$, the map $\mu^{(b)} (u) (\cdot)$ is a Radon measure.
\end{proposition}
\begin{proof}
 We only have to show, that $\mu^{(b)}$ is inner regular, the rest is obvious. For this let $E \subset X$ be measurable. As $J$ is a Radon measure, $\mu^{(b)}(u )(E)$ can be approximated by $\int_{K} (\tilde{u}(x) - \tilde{u}(y))^2 \: J(\rmd x, \rmd y)$ with $K \subset E \times X - d$ compact. But then $\mu^{(b)}(u)(K') = \int_{K' \times X - d} (\tilde{u}(x) - \tilde{u}(y))^2 \: J(\rmd x, \rmd y)$ with $K' := \{x \in X : \exists y \in X \text{ with } (x, y) \in K\}$ the projection from $K$ on the first component also approximates $\mu^{(b)}(u )(E)$. As $K'$ is compact the desired regularity follows.
\end{proof}

Having extended $\mu^{(b)}$ and $\mu^{(c)}$ we can also extend their sum and obtain a Radon measure again denoted by $\mu^{(d)}$.
Again, we call $\mu^{(b)}$ the \emph{jump part} or the \emph{pseudo-differential part} or the \emph{non-local part of the energy measure} and $\mu^{(d)}$ the \emph{differential part of the energy measure}.

\smallskip

Now, we can also extend the form $\calE$ in the desired way.

\begin{theorem} Any $\varphi \in \calD_{loc}^*$ with compact support belongs to $\calD$ and
 \[ \calE(u, \varphi) := \int_X \: \rmd \mu^{(d)}(u, \varphi) + \int_X \tilde{u} \tilde{\varphi} \: \rmd k \]
 is well-defined (i.e., the integrals on the right hand side exist) for all $u, \varphi \in \calD_{loc}^*$ with $\supp \varphi$ compact.
\end{theorem}
\begin{proof}
 We only have to look at the non-local part, that is
 \[ \int_{X \times X - d} (\tilde{u}(x) - \tilde{u}(y)) (\tilde{\varphi}(x) - \tilde{\varphi}(y)) \: J(\rmd x, \rmd y). \]
 We now set $K := \supp \tilde{\varphi}$. Because the integral over $K \times X - d$ is well-defined, and over $K^c \times K^c - d$ is zero, we only have to look at $K^c \times K - d$ or, using the symmetry, at $K \times K^c - d$. But on this set the integral is well-defined as part of a well-defined integral.
\end{proof}

\smallskip

The space $\calD_{loc}^*$ and the measures $\mu^{(*)}$ are well compatible with approximation via cut-off procedures.  This will be relevant later on. We discuss the corresponding details in the next lemma.

\begin{lemma}\label{lem:3.5}
 Let $u \in L^\infty_{loc}\cup \calD_{loc}$ and assume that there is a Radon measure $m_1$ such that for every $T>0$ one has $u_T := ( u \wedge T) \vee( -T) \in \calD_{loc}^*$ and $\mu^{(b)}(u_T) \leq m_1$. Then $u \in \calD_{loc}^*$ and $ \mu^{(*)}(u) = \lim_{T \rightarrow \infty} \mu^{(*)}(u_T )$ for $*=a,b,c,d$. In particular, $\mu^{(b)}(u) \leq m_1$.
\end{lemma}

\begin{proof}
 Note that $u\in L_{loc}^\infty $ agrees locally with $u_T$ for $T$ big enough. Thus, obviously $u$ belongs to $\calD_{loc}$. Therefore we only have to show that
 \[ \int_{K \times X - d} (\tilde{u}(x) - \tilde{u}(y))^2 \: J(\rmd x, \rmd y) < \infty. \]
 This follows as $(u_T (x) - u_T (y))^2$ converges monotonically to $(u(x) - u(y))^2$ and
 \[ \int_{K \times X - d} (u_T(x) - u_T(y))^2 \: J(\rmd x, \rmd y) \leq m_1(K) < \infty \]
 uniformly in $T$. For $*=a,b$ the convergence of $\mu^{(*)}(u_T )$ is also clear by monotone convergence. To deal with $* = c$ , i.e., the strongly local part, we note that $v_T \rightarrow v$ with respect to $\calE_1$ for all $v \in \calD$; see \cite[Theorem 1.4.2.iii]{FukushimaOT-94}.
\end{proof}

\subsection{The Leibniz rule} \label{Leibnizrule}

We now turn to stability under taking products and the Leibniz rule. In order to do this we introduce the measure $\Gamma$. Let $C_c (X \times X)$ denote the continuous real valued functions on $X \times X$ with compact support.
For $u,v\in \calD_{loc}^* $ we define the measure $\Gamma(u,v)$ on $X\times X$ by
\begin{footnotesize}
 \[ \sint{X \times X} f(x, y) \: \rmd \Gamma(u, v) := \int_X f(x, x) \: \rmd \mu^{(c)}(u, v) + \sint{X \times X - d} f(x, y) (\tilde{u}(x) - \tilde{u}(y)) (\tilde{v}(x) - \tilde{v}(y)) \: J(\rmd x, \rmd y) \]
\end{footnotesize}
for $f\in C_c (X\times X)$. This is well defined by the Definition of $\calD_{loc}^*$ and Cauchy-Schwarz inequality.
 We are going to extend this equality to a larger class of functions. These are the functions defined quasi everywhere. They are given as follows:
 A measurable $f : X \times X \rightarrow \RR$ is said to be defined q.e., if there exists a set $E\in X$ with $\capp (E)=0$ and $f$ is defined on $((X \setminus E) \times X) \cap (X \times (X \setminus E)) = (E \times E)^c$. The following lemma will allow us to extend the definition of $\Gamma$ to functions which are defined q.e.

\begin{lemma}
 Let $u \in \calD_{loc}^*$ and $M \subset X$ be measurable with $\capp(M) = 0$. Then
 \[ \sint{M \times X - d} (\tilde{u}(x) - \tilde{u}(y))^2 \: J(\rmd x, \rmd y) = 0. \]
\end{lemma}
\begin{proof}
 As $J$ is a Radon measure and thus inner regular it suffices to show the claim for $M$ compact. As $M$ is compact with $cap (M)=0$, there is a sequence $f_n \in \calD \cap C_c(X)$ with $f_n = 1$ on $M$, $f_n \geq 0$ and $\calE_1(f_n, f_n) \rightarrow 0$. For any $v\in L^\infty \cap \calD$ we then have by the results on algebraic structure above that $ \calE_1 (v f_n)$ is bounded. In particular, $v f_n $ contains $\calE_1$ weakly converging subsequences. As $f_n$ converges to $0$ in $L^2$ and $v$ is bounded, the sequence $(v f_n)$ converges to $0$ in $L^2$ and we infer that $v f_n$ itself converges $\calE_1$ weakly to $0$. After these preparations we can now proceed as follows: By a standard approximation result we can replace
 \[ \sint{M \times X - d} (\tilde{u}(x) - \tilde{u}(y))^2 \: J(\rmd x, \rmd y) < \infty \]
 by the integral
 \[ \sint{M \times K - d} (\tilde{u}_T(x) - \tilde{u}_T(y))^2 \: J(\rmd x, \rmd y) \]
 with $K$ compact and $u_T := u \wedge T \vee (-T)$. We now choose $v \in \calD \cap L^\infty$ with $v = u_T$ on $K \cup M$. Then
 \begin{align*}
 \sint{M \times K - d} (\tilde{u}_T(x) - \tilde{u}_T(y))^2 \: J(\rmd x, \rmd y) &= \sint{M \times K - d} (\tilde{v}(x) - \tilde{v}(y))^2 \: J(\rmd x, \rmd y) \\
 &\leq \sint{X \times X - d} f_n(x) (\tilde{v}(x) - \tilde{v}(y))^2 \: J(\rmd x, \rmd y) \\
 &= \calE^{(b)}(v f_n, v) - \frac{1}{2} \calE^{(b)} (f_n, v^2).
 \end{align*}
 As discussed above, $(f_n)$ tends to $0$ with respect to $\calE_1$ and $(v f_n)$ converges $\calE_1$ weakly to $0$ and we see that the right hand side of the last inequality tends to zero. This finishes the proof.
\end{proof}

For $u, v \in \calD_{loc}^*$ and $f$ measurable and defined q.e.\ we then obtain
\begin{footnotesize}
 \[ \sint{X \times X} f(x, y) \: \rmd \Gamma(u, v) = \int_X f(x) \: \rmd \mu^{(c)}(u, v) + \sint{X \times X - d} f(x, y) (\tilde{u}(x) - \tilde{u}(y)) (\tilde{v}(x) - \tilde{v}(y)) \: J(\rmd x, \rmd y). \]
\end{footnotesize}

Given this, we can note the following version of the Leibniz rule.

\begin{theorem}[Leibniz rule]\label{Leibniz}
 Let $u, v \in \calD_{loc}^*$ with $u \cdot v \in \calD_{loc}^*$ be given. Then the following holds:
 \begin{itemize}
 \item [(a)] The equality
  \[ \sint{X \times X} f(x, y) \: \rmd \Gamma(u \cdot v, w) = \sint{X \times X} f(x, y) \tilde{u}(x) \: \rmd \Gamma(v, w) + \sint{X \times X} f(x, y) \tilde{v}(y) \: \rmd \Gamma(u, w) \]
  holds for all $f$ measurable and defined q.e., whenever (at least) two of the integrals exist.
 \item[(b)] The equation
  \[ \int_X \: \rmd \mu^{(d)}(u \cdot v, w) = \int_X \tilde{u} \: \rmd \mu^{(d)}(v, w) + \int_X \tilde{v} \: \rmd \mu^{(d)}(u, w) \]
  holds, whenever (at least) two of the integrals exist.
 \end{itemize}
\end{theorem}
\begin{proof}
 (a) We consider the strongly local and the jump part separately.

 \smallskip

 For the strongly local part and bounded $u$ and $v$ this is well known (see e.g.\ the discussion in \cite{Sturm-94b}). With the truncation property we can then conclude for $u, v$ with $u \cdot v \in \calD_{loc}^*$ that
 \begin{equation}\label{eq:LeibnizSL}
 \mu^{(c)}(u v, w) = \tilde{u} \mu^{(c)}(v, w) + \tilde{v} \mu^{(c)}(u, w)
 \end{equation}
 on $F_n := \{ | \tilde{u} | \leq n \text{ and } | \tilde{v} | \leq n \}$ for all $n \in \NN$ and hence on $F := \bigcup_{n=1}^\infty F_n$. By \cite[Lemma 2.1.6]{FukushimaOT-94} $\capp(\{ |w| > n \}) \leq \frac{\calE_1(w)}{n^2}$ for every $w \in \calD$. As both $u$ and $v$ agree on compact sets with elements from $\calD$ we can then infer that $\capp(F^c \cap K) = 0$ for every compact $K \subset X$. As $\capp$ is a Choquet capacity $\capp(F^c) = 0$ follows. Now, the equality (\ref{eq:LeibnizSL}) follows on the full set $X$, as $\mu^{(c)}$ does not charge sets of capacity zero, see \cite[Lemma 3.2.4.]{FukushimaOT-94}.

 The non-local part can be treated by simple algebraic manipulations involving
 \[ (u \cdot v)(x) - (u \cdot v)(y) = u(x) \cdot (v(x) - v(y)) + v(y) \cdot (u(x) - u(y)). \]

 \smallskip

 (b) This is a direct consequence of (a).
\end{proof}

To apply the Leibniz rule we need that the product of the two functions $u,v \in \calD_{loc}^* $ belongs again to this space. Here, we discuss sufficient conditions for this (compare Theorem \ref{thm:algebraic_structure} as well.)

\begin{proposition}\label{thm:products}
 \begin{enumerate}
 \item[(a)] Let $u \in \calD_{loc}^* \cap L_{loc}^\infty(X)$ and $v \in \calD_{loc}^* \cap L^\infty(X)$. Then $u v \in \calD_{loc}^* \cap L_{loc}^\infty(X)$.
 \item[(b)] Let $u \in \calD$ and $v \in \calD_{loc}^* \cap L^\infty(X)$ be given such that $\mu^{(d)}(v)$ is absolutely continuous with respect to $m$ with bounded density. Then $u v \in \calD$.
 \item[(c)] Let $u \in \calD_{loc}^*$ and $v \in \calD_{loc}^* \cap L^\infty(X)$ be given such that $\mu^{(d)}(v)$ is absolutely continuous with respect to $m$ with bounded density. Then $u v \in \calD_{loc}^*$.
 \end{enumerate}
\end{proposition}

\begin{proof}
 (a) From the algebraic properties, we see, that $u \cdot v \in \calD_{loc}$. For the rest see
 \begin{footnotesize}\begin{align*}
 \sint{K \times X - d} &((u \cdot v)(x) - (u \cdot v)(y))^2 \: J(\rmd x, \rmd y) = \sint{K \times X - d} u(x)^2 (v(x) - v(y))^2 + ...\\
 &\qquad ... + 2 u(x) v(y) (u(x) - u(y)) (v(x) - v(y)) + v(y)^2 (u(x) - u(y))^2 \: J(\rmd x, \rmd y) \\
 &\leq 2 \| u \|_{\infty, K}^2 \sint{K \times X - d} (v(x) - v(y))^2 \: J(\rmd x, \rmd y) + 2 \| v \|_\infty^2 \sint{K \times X - d} (u(x) - u(y))^2 \: J(\rmd x, \rmd y) < \infty.
 \end{align*}\end{footnotesize}

 (b) Let $u_n := u \wedge n \vee -n$. Let $C$ be a bound for $v^2$ and for the density $\rmd \mu^{(d)}(v, v) / \rmd m$. We then have
 \begin{align*}
 \calE(u_n v, u_n v) &= \calE^a(u_n v, u_n v) + \sint{X \times X} \:\rmd\Gamma(u_n v, u_n v) \\
 \intertext{and by Leibniz rule and Cauchy-Schwarz}
 &\leq C \calE^a(u_n, u_n) + 2 \sint{X \times X} u_n(x)^2 \:\rmd\Gamma(v, v) + 2 \sint{X \times X} v(y)^2 \:\rmd\Gamma(u_n, u_n) \\
 &\leq 2 C \calE_1(u_n, u_n) \leq 2 C \calE_1(u)
 \end{align*}
 uniformly in $n$. As obviously $(u_n v)$ converges to $u v$ in $L^2$ we conclude $u v \in \calD$ by closedness of the form (see e.g.\ \cite[Theorem VI.1.16]{Kato-66}).

 (c) From (b) we know, that $u v \in \calD_{loc}$. Now calculate (for details see (a) and (b))
 \begin{align*}
 \sint{K \times X - d} ((u \cdot &v)(x) - (u \cdot v)(y))^2 \: J(\rmd x, \rmd y) \leq 2 \sint{K \times X - d} u(x)^2 (v(x) - v(y))^2 \: J(\rmd x, \rmd y) + ... \\
 &\qquad ... + 2 \sint{K \times X - d} v(y)^2 (u(x) - u(y))^2 \: J(\rmd x, \rmd y) \\
 &\leq 2 \int_K u^2 \:\rmd\mu^{(b)}(v, v) + 2 \| v \|_\infty^2 \int_K \:\rmd\mu^{(b)}(u) < \infty. \qedhere
 \end{align*}
\end{proof}

The previous discussion naturally raises the question whether the product of $u\varphi$ belongs to $\calD$ whenever $u\in \calD $ and $\varphi\in \calD \cap C_0 (X)$. This does not need to be the case as we illustrate with both a non-local and a local example:

\begin{example}
 Let $X := [-1, 1]$ and
 \[ \calE(u, v) := \int_0^1 (u(x) - u(-x))(v(x) - v(-x)) |x|^{-5/2} \: \rmd x \]
 \[ \tilde{\calD} := \{u \in C(X) \mid \calE(u) < \infty \} \]
 By Fatou's lemma the form $\calE$ is closable. Let $(\calE, \calD)$ be its closure. By definition $(\calE, \calD)$ is then a regular Dirichlet form. Note that $\calE (w)$ vanishes whenever $w$ is even. Simple cut-off procedures then show that $v$ with $v(x) := |x|^{-1/4}$ belongs to $\calD$. Moreover, $u$ with
 $u(x) := x$ belongs to $\widetilde{\calD}\subset \calD$. Hence $u, v \in \calD$, but by a direct calculation one sees that the product $u v $ does not belong to $\calD$.
\end{example}

\begin{example}
 Define
 \[ u(x) := \begin{cases}
       |x|^{5/8} \cdot \sin(\frac{1}{|x|}) & |x| \in (0, \pi^{-1}) \\
       0 & \text{otherwise}
      \end{cases} \]
 \[ v(x) := (|x|^{-1/4} - 1)^+. \]
 Then direct calculations show that $u \in W^{1, 2}(\RR^3) \cap C_c(\RR^3)$ and $v \in W^{1, 2}(\RR^3)$, but $u \cdot v \notin W^{1, 2}(\RR^3)$ as the derivative of $u v$ does not belong to $L^2$.
\end{example}

\section{Intrinsic metrics}\label{Intrinsicmetric}

For strongly local Dirichlet forms the intrinsic metric is a powerful tool. It has been used in studying decay of heat kernels, the investigation of Harnack inequalities and to get 'good' cut-off functions in the study of spectral properties e.g.\ in \cite{BiroliM-95, BoutetdeMonvelLS-08, Sturm-94b, terElstRSZ-06}.
Our aim is to generalize this concept to non-local Dirichlet forms. This is done in this section.

\medskip

We begin with a short discussion of pseudo-metrics. A map $\varrho : X \times X \rightarrow [0,\infty]$ is called a \emph{pseudo-metric} if $\varrho(x,x)=0$, $\varrho (x,y) = \varrho (y,x)$ and $\varrho(x,y) \leq \varrho (x,z) + \varrho (z,y)$ for all $x,y,z\in X$. A consequence of the triangle inequality, which will be useful later on, is the estimate
\[ |\varrho (x,y) - \varrho(x',y')|\leq \varrho(x,x') +\varrho (y,y'). \]
We emphasize that $\varrho$ may not be continuous with respect to the original topology.

Whenever $\varrho$ is a pseudo-metric on $X$ and $A \subset X$ we can define
\[ \varrho_A(x) := \inf_{y \in A} \varrho(x, y). \]
If $\varrho$ is a pseudo-metric, then so is $\varrho \wedge T$ for any $T \geq 0$. One has
\[ (\varrho \wedge T)_A = \varrho_A \wedge T. \]
and the estimate
\[ |\varrho_A (x) \wedge T - \varrho_A (y) \wedge T| \leq \varrho (x,y) \]
holds for any $x, y \in X$. This estimate shows that, if $\varrho$ is continuous, then so is $\varrho_A\wedge T$ and $\varrho_A$.

\begin{definition}
 A pseudo-metric $\varrho : X \times X \rightarrow [0, \infty]$ is called an \emph{intrinsic metric} with respect to the Dirichlet form $\calE$ if there are two Radon measures $m_b$ and $m_c$ with $m_b + m_c \leq m$ such that for all $A\subset X$ and all $T > 0$ the function $\varrho_A$ defined above satisfies
 \begin{itemize}
 \item $\varrho_A \wedge T \in \calD_{loc}^* \cap C(X)$,
 \item $\mu^{(b)}(\varrho_A \wedge T) \leq m_b$,
 \item $\mu^{(c)}(\varrho_A \wedge T) \leq m_c$.
 \end{itemize}
\end{definition}

To illustrate this notion we consider the Dirichlet forms of the continuous and discrete Laplacians.

\begin{example}
Consider $X=\RR^d$, $d\geq 1$, with Lebesgue measure and
\[ \calE(u) := \int_{\RR^d} |\nabla u|^2 dx, \qquad \calD=W^{1,2}(\RR^d). \]
Then the standard Euclidean distance, $\varrho(x,y):=|x-y|$, is an intrinsic metric for $\calE$. Indeed, for $A\subset \RR^d$ and $T> 0$ the function $\varrho_A \wedge T$ is Lipschitz continuous and its gradient exists a.e.\ and equals $|\nabla (\varrho_A \wedge T)|=1$ on $\{\varrho_A <T\}$ and $=0$ on $\{\varrho_A \geq T\}$. Therefore the conditions in the definition are satisfies with $m_b=0$ and $m_c=$ Lebesgue measure.
\end{example}

\begin{example}\label{ex:disclapl}
Consider $X=\ZZ^d$, $d\geq 1$, with counting measure and
\[ \calE(u) := \sum_{|x-y|=1} |u(x)-u(y)|^2, \qquad \calD=\ell^2(\ZZ^d). \]
(Here $|x-y|$ denotes the distance induced from that in $\RR^d$.) We claim that $\varrho(x,y):=(1/\sqrt{2d})|x-y|$ defines an intrinsic metric for $\calE$ with $m_b=$ counting measure and $m_c=0$. To prove this, we note that if $|x-z|=1$ then $|\varrho_A(x) \wedge T - \varrho_A(z)\wedge T|\leq |\varrho_A(x) - \varrho_A(z)| \leq 1/\sqrt{2d}$. Since any $z$ has $2d$ neighbors $x$, we conclude that for any $z$, $\mu^{(b)}(\varrho_A\wedge T)(\{z\}) = \sum_{|x-z|=1} |\varrho_A(x) \wedge T - \varrho_A(z)\wedge T|^2 \leq 1$,
which proves the claim.
\end{example}

We emphasize that we require $\varrho_A \wedge T$ for an intrinsic metric $\rho$ to be continuous with respect to the topology generated by the underlying metric $d$, but we do not require $d_A \wedge T$ to be continuous with respect to the topology generated by $\varrho$. In general these two topologies do not coincide; see Example \ref{ex:topo} below.

Let us collect some simple properties of intrinsic metrics.

\begin{proposition}\label{distprop}
 Let $\varrho$ be an intrinsic metric and let $A\subset X$ be such that $\varrho_A(x) < \infty$ for all $x \in X$. Then $\varrho_A \in \calD_{loc}^* \cap C(X)$ and $\mu^{(d)}(\varrho_A) \leq m$.
\end{proposition}

\begin{proof}
The continuity of $\varrho_A$ follows from the continuity of $\varrho_A\wedge T$ for any $T$. By the definition of an intrinsic metric, we have $\varrho_A \wedge T \in \calD_{loc}^*$ and $\mu^{(d)}(\varrho_A \wedge T) \leq m_b+m_c=m$ for any $T>0$. Hence the assertion follows using Lemma \ref{lem:3.5}.
\end{proof}

The next definition is standard.
\begin{definition} \label{cutoff} Let $\varrho$ be an intrinsic metric. Let $E\subset X$ and $a > 0$ be given.
 \begin{itemize}
 \item[(a)] The \emph{cut-off function} associated to $E$ with range $a$ is given by
\[ \eta_{E, a}(x) := (1 - \varrho_E(x) / a)^+. \]
 \item[(b)] The \emph{intrinsic ball} around $E$ with radius $a$ is given by
  \[ B_r(E) := \{x \in X : \varrho_E(x) \leq r \}. \]
 \item[(c)] The \emph{intrinsic boundary} of a set $E$ is given by
 \[ A_r(E) := B_r(E) \cap B_r(E^c). \]
 \end{itemize}
\end{definition}

\begin{proposition}\label{cutoffprop}
Let $\varrho$ be an intrinsic metric, $E\subset X$ and $a>0$. Then $\eta_{E, a}\in \calD_{loc}^* \cap C(X)$ and $\mu^{(d)}(\eta_{E, a}) \leq (1/a^2) m$. Moreover, if $B_a(E)$ is relatively compact, then $\eta_{E, a} \in \calD \cap C_c(X)$.
\end{proposition}

\begin{proof} Since $\rho_E$ is continuous, $\eta_{E,a}$ is so as well. Moreover, $\rho_E$ belongs to $\calD_{loc}^*$ by Proposition \ref{distprop}, and as Dirichlet form, $\calE$ is compatible with cut-off procedures. Hence $\eta_{E,a}\in \calD_{loc}^*$. In order to show the claimed upper bound on $\mu^{(d)}(\eta_{E, a})$ we recall that $\mu^{(c)}(\eta_{E, a}) \leq (1/a^2) \mu^{(c)}(\rho_E)$ \cite[Theorem 3.2.2.]{FukushimaOT-94} Moreover, since $|\eta_{E,a}(x)-\eta_{E,a}(y)| \leq (1/a) |\rho_E(x)-\rho_E(y)|$ we have
$\mu^{(b)}(\eta_{E,a}) \leq (1/a^2) \mu^{(b)}(\rho_E)$. Therefore, the bound $\mu^{(d)}(\varrho_E) \leq m$ from Proposition \ref{distprop} implies the bound $\mu^{(d)}(\eta_{E, a}) \leq (1/a^2) m$. The last statement is obvious.
\end{proof}

\medskip

Before presenting different methods for finding an intrinsic metric in the next section, we exhibit two useful results about intrinsic metrics. They will play an important role later when we deal with spectral theory.

\begin{lemma}\label{lem:4.8}
 Let $\varrho$ be an intrinsic metric. Then
 \[ \sint{E \times X - d} \varrho^2(x, y) \: J(\rmd x, \rmd y) \leq m_b(E) \]
for any measurable set $E\subset X$.
\end{lemma}
\begin{proof}
 Let $\varepsilon > 0$ and $s > 2 \varepsilon$ be arbitrary. We first consider sets $E$ with $E \subset B_\varepsilon(\tilde{x})$ for some $\tilde{x}$. Using the fact that for $x\in E$
 \[ \varrho(x,y) \leq \varrho(y,\tilde x) - \varrho(x,\tilde x) + 2 \varrho(x,\tilde x) \leq |\varrho(y,\tilde x) - \varrho(x,\tilde x) | + 2 \varepsilon \]
 we can estimate for every $\delta>0$
 \[ \sint{\substack{E \times X \\ \varrho(x, y) > s}} \varrho^2(x, y) \: J(\rmd x, \rmd y)
\leq (1 + \delta) \sint{E \times X - d} (\varrho(x, \tilde{x}) - \varrho(y, \tilde{x}))^2 \: J(\rmd x, \rmd y) + (1 + \frac{1}{\delta})\, 4 \varepsilon^2 \sint{\substack{E \times X \\ \varrho(x, y) > s}} \rmd J. \]
 The first term on the right side is controlled since by the definition of an intrinsic metric and by Lemma \ref{lem:3.5} we have
 \[ \sint{E \times X - d} (\varrho(x, \tilde{x}) - \varrho(y, \tilde{x}))^2 \: J(\rmd x, \rmd y) \leq m_b(E). \]
 In order to controll the second term on the right side we estimate for $x\in E$ and $y\in X$ with $\varrho(x,y)>s$
 \[ \varrho(y,\tilde x) - \varrho(x,\tilde x) \geq \varrho(y,x) - 2 \varrho(x,\tilde x) \geq s- 2\epsilon, \]
 which yields
 \[ \sint{\substack{E \times X \\ \varrho(x, y) > s}} \rmd J
 \leq \frac{1}{(s - 2 \varepsilon)^2} \int_{E \times X - d} (\varrho(x, \tilde{x}) - \varrho(y, \tilde{x}))^2 \: J(\rmd x, \rmd y). \]
 Putting these estimates together we infer that
 \begin{align*}
 \sint{\substack{E \times X \\ \varrho(x, y) > s}} \varrho^2(x, y) \: J(\rmd x, \rmd y)
 \leq \left(1 + \delta + \left(\frac{2 \varepsilon}{s - 2 \varepsilon}\right)^2 (1 + \frac{1}{\delta})\right) m_b(E).
 \end{align*}

 With this estimate at hand, we can now pass to arbitrary compact sets $E$. An arbitrary compact $E$ can be covered by finitely many disjoint sets $E_n$, each one being contained in an intrinsic ball of radius $\epsilon$. (Indeed, by compactness $E$ can be covered by finitely many of the balls $B_\epsilon(\{x\})$, $x\in E$; now make these disjoint by removing from the $n$-th set the union of the first $n-1$ as well as $E^c$.) In this way, the previous estimate extends to arbitrary compact $E$. Letting first $\varepsilon \to 0$, then $\delta \to 0$ and finally $s\to 0$ we obtain the desired estimate for all compact sets $E$. The general case follows from regularity.
\end{proof}

The next theorem, sometimes referred to as Rademacher Theorem, is well known for the usual Sobolev spaces on Euclidean space. For general strongly local Dirichlet forms it seems to be new. It has already proven useful in Stollmann's work on path spaces associated to Dirichlet forms \cite{Stollmann-09}. For related material in the context of some infinite dimensional spaces we refer to \cite{RoecknerSch-99, Schied-02}.

\begin{theorem}\label{thm:cool}
 Let $\varrho$ be an intrinsic metric. Then every $u: X \rightarrow \RR$ with $|u(x) - u(y)| \leq \varrho(x, y)$ satisfies $u \in \calD_{loc}^*$ and $\mu^{(d)}(u) \leq m$.
\end{theorem}
\begin{proof}
 As a preliminary observation we note that for any relatively compact $G \subset X$, there exists  a compact $K\supset G$ with
\[ \int_{G\times K^c} dJ (x,y)< \infty.\;\:\; (*) \]
(Just use finiteness of $\calE(\phi)$ for $\phi \in C_c (X)\cap \calD$ with $\phi \equiv 1$ on $G$ and set $K$ to be the support of $\phi$.)

\medskip

Let us now turn to the actual proof. By Lemma \ref{lem:3.5} we can assume $0 \leq u \leq M$. Define
\[ u^*_n := \max_{m=1,\ldots,(n \cdot M)} \left( \frac{m}{n} - \varrho_{\{x \mid u(x) \geq \frac{m}{n} \}} \right)^+. \]
 Then $u^*_n$ belongs to $ \calD_{loc}^*\cap C (X)$ as it is a maximum of finitely many functions in $\calD_{loc}^*\cap C (X)$ and it converges to $u$ with respect to the supremum norm by assumption on the continuity property of $u$. Moreover, we have $\mu^{(c)}(u^*_n) \leq m_c$ by the truncation property and we have $\mu^{(b)}(u^*_n) \leq m_b$ by Lemma \ref{lem:4.8} (as $|u^*_n (x) - u^*_n (y)| \leq \varrho (x,y)$ by the very definition of $u^*_n$).

\medskip

We will show that $u$ belongs to $\calD_{loc}$ and that $\mu^{(d)}(u) \leq m$ on $G$. (This will then implicitely give that $u\in \calD_{loc}^*$ as well.) Let $G \subset X$ be open and relatively compact. For $\phi \in C_c(X) \cap \calD$ with $\phi \equiv 1$ on $G$ define $u_n := u^*_n \wedge M \phi$.

\medskip

\textit{Claim.} The functions $u_n$ belong to $\calD$ and $\calE_1(u_n)$ is bounded in $n$.

Proof of claim. It suffices to show uniform (in $n$) boundednes of $\mu^{(*)} (u_n) (X)$ for $*=a,b,c,d$:
Note that the $u_n$ are uniformly bounded by $M$ and have all support contained in the compact support of $\phi$. This easily gives the desired boundedness for $* = a$ and by the truncation property it also gives the desired boundedness for $* = c$. It remains to consider the case $* = b$.
 Choose a compact $K\subset X$ with $\supp \phi\subset K$ according to $(*)$ with
\[ \int _{K^c \times \supp \phi} d J (x,y) = C <\infty. \]
Then, we can estimate
\begin{footnotesize}
 \[ \sint{X \times X} (u_n(x) - u_n(y))^2 \:J(\rmd x, \rmd y) = \sint{K \times X} (u_n(x) - u_n(y))^2 \:J(\rmd x, \rmd y) + \sint{K^c \times X} (u_n(x) - u_n(y))^2 \:J(\rmd x, \rmd y). \]
\end{footnotesize}
We estimate the two terms on the right hand side. As $u_n$ is the minimum of $u_n^*$ and $M\phi$, we have
\[ (u_n (x) - u_n (y))^2 \leq 2 (u_n^* (x) - u_n^* (y))^2 + 2 (M \phi (x) - M \phi (y))^2 \]
and can hence bound the first term above by
\begin{small}
\[ 2 \sint{K \times X} (u_n^*(x) - u_n^*(y))^2 \:J(\rmd x, \rmd y) + 2 M^2 \sint{K \times X} (\phi(x) - \phi(y))^2 \:J(\rmd x, \rmd y) \leq 2 m(K) + 2 M^2 \calE (\phi) \]
\end{small}
independently of $n$. Moreover, $\supp u_n \subset \supp \phi \subset K$ and hence the second term can be bounded above by
\begin{align*}
 \int_{K^c \times X} (u_n (x) - u_n (y))^2 d J (x,y) &= \int_{K^c \times X} u_n (y)^2 d J(x,y) \\
 &= \int_{K^c \times \supp \phi} u_n (y)^2 dJ (x,y) \\
 &\leq M \int_{K^c \times \supp \phi} dJ (x,y) = M C
\end{align*}
independently of $n$. This finishes the proof of the claim.

\bigskip

 As $(u_n)$ is bounded with respect to $\calE_1$, it has a weakly $\calE_1$-converging subsequence. As $(u_n)$ obviously converges to $v = u \wedge M\phi$ in $L^2$ we infer $v\in \calD$ as well as weak $\calE_1$-convergence of $(u_n)$ to $v$. Since the relatively compact $G$ and $\phi$ with $\phi\equiv 1$ on $G$ are arbitary, this shows, in particular, that $u$ belongs to $\calD_{loc}$.

\smallskip

 It remains to show that $\mu^{(d)}(u) \leq m$ on $G$. First we have $\mu^{(b)}(u) \leq m_b$ by Lemma \ref{lem:4.8}, since $|u(x) - u(y)| \leq \varrho (x,y)$.

 Choose $\phi \in C_c(X) \cap \calD$ with $\phi \equiv 1$ on $G$. Define $u_n$, $v$ as before. Choose some test function $f \in C_c(G) \cap \calD$, $f \geq 0$.
 Since $u_n$ agrees locally with cut-off functions, we obtain by the truncation property
 \[ \int f \rmd \mu^{(c)}(u_n) \leq \int f \rmd m_c. \]
 Finally, we obviously have $0 \leq\int f \rmd \mu^{(c)} (u_n - v)$ and hence
 \[ 2 \int f \rmd \mu^{(c)}(u_n, v) - \int f \rmd \mu^{(c)}(v, v) \leq \int f \rmd \mu^{(c)}(u_n). \]
 As the $u_n$ converge $\calE_1$ weakly to $v$ we can now use Proposition \ref{prop:continuity} to obtain by putting these estimates together
 \begin{align*}
 \int f \rmd \mu^{(c)}(u) &\leq \int f \rmd \mu^{(c)}(v) = \lim_{n \rightarrow \infty} \left( 2 \int f \rmd \mu^{(c)}(u_n, v) - \int f \rmd \mu^{(c)}(v)\right) \\
  &\leq \liminf_{n \rightarrow \infty} \int f \rmd \mu^{(c)}(u_n) \leq \int f \rmd m_c. \\
 \end{align*}
 This gives $\mu^{(c)}(u) \leq m_c$ and the desired result.
\end{proof}

The theorem gives a stability property of intrinsic metrics.

\begin{corollary}\label{cor:IM_order}
 Let $\varrho_1$ be a pseudo-metric, $\varrho_2$ be an intrinsic metric, and $\varrho_1 \leq \varrho_2$. Then $\varrho_1$ is an intrinsic metric.
\end{corollary}
\begin{proof} The triangle inequality and the assumption imply that $\varrho_1$ satisfies
\[ |\varrho_{1,A} (x) - \varrho_{1,A} (y)|\leq \varrho_{1} (x,y) \leq \varrho_2 (x,y). \]
Now the desired statement follows from the previous theorem.
\end{proof}

\section{Intrinsic metrics and sets of intrinsic-metric functions}\label{Intrinsic-metric-functions}

In this section we show that specifying an intrinsic metric is essentially equivalent to specifying a suitable set of (Lipschitz) continuous functions.

\medskip

We start by introducing the relevant sets of functions.
\begin{definition}
 A set $M \subset \calD_{loc}^* \cap C(X)$ is called \emph{set of intrinsic-metric functions} if there are two Radon measures $m_b$ and $m_c$ with $m_b + m_c \leq m$ such that the following conditions are satisfied.
 \begin{itemize}
 \item $0 \in M$.
 \item $\mu^{(b)}(f) \leq m_b$ and $\mu^{(c)}(f) \leq m_c$ for all $f \in M$,
 \item $f+c \in M $ and $-f \in M$ for all $c\in \RR$ and $f\in M$.
 \item $\sup\{f_i\} \in M$, whenever $f_i \in M$, $i \in \NN$ with $\sup\{f_i\}$ finite.
 \end{itemize}
 For such a set $M$ we define
 \[ \varrho^{(M)} (x,y) := \sup_{f \in M \,\text{and}\, f(y) = 0} f(x). \]
\end{definition}

\begin{remark}
 $\varrho^{(M)}$ is a pseudo-metric. Note also that $\varrho^{(M)}$ is the smallest pseudo-metric such that any $f\in M$ satisfies $|f(x)-f(y)| \leq \varrho^{(M)}(x,y)$ for any $x,y\in X$.
\end{remark}

The main result of this section is

\begin{theorem}\label{thm:3.9}
 If $\varrho$ is an intrinsic metric, then
 \[ M := \{f : X \rightarrow \RR \mid \; |f(x) - f(y)| \leq \varrho(x, y) \} \]
 is a set of intrinsic-metric functions with $\varrho^{(M)}=\varrho$.
 Conversely, if $M$ is an arbitrary set of intrinsic-metric functions, then $\varrho^{(M)}$ is an intrinsic metric.
\end{theorem}

The first part of this theorem is a consequence of Theorem \ref{thm:cool}. The second part follows from the next two lemmas. We abbreviate $\varrho^{(M)}_A := (\varrho^{(M)})_A$ as defined in the previous section.

\begin{lemma} Let $M$ be a set of intrinsic-metric functions. Then
 \[ \varrho^{(M)}_A (x) = \sup_{f \in M \,\mathrm{and}\, f(A) = 0} f(x). \]
\end{lemma}

\begin{proof} We denote the right side by $\varrho^{(M)}(x,A)$.\\
 ``$\geq$'': We have $\varrho^{(M)}(x, A) \leq \varrho^{(M)}(x, y)$ for all $y \in A$. \\
 ``$\leq$'': Let $T := \varrho^{(M)}_A (x) = \inf_{y \in A} \varrho^{(M)}(x, y)$. Then $g := (T - \varrho^{(M)}(\cdot, x)) \vee 0$
satisfies $g\in M$ and $g(A)=0$. Hence using $g$ in the supremum that defines $\varrho^{(M)}(x,A)$, we find
$\varrho^{(M)}(x,A)\geq g(x)=T = \varrho^{(M)}_A (x)$.
\end{proof}

\begin{lemma}
 Let $M$ be a set of intrinsic-metric functions. Then $\varrho^{(M)}$ is continuous and $\varrho^{(M)}_A \wedge T\in M$ for all $T \in \RR$. Moreover, if $\varrho^{(M)}_A$ takes only finite values, then $\varrho^{(M)}_A$ belongs to $M$.
\end{lemma}

\begin{proof}
We begin by proving that for any $x$, the function $y\mapsto \varrho^{(M)}(y,x)$ is continuous at $x$. Assume not, then there is an $\epsilon>0$ and a sequence $y_i \rightarrow x$ with $\varrho^{(M)}(y_i,x) > 2 \epsilon$. Then there are $f_i \in M$ with $f_i(x) = 0$ and $f_i(y_i)>\epsilon$. Since $M$ is compatible with truncations we may assume that $f_i \leq 2 \epsilon$. From this $f := \sup f_i \in M$ follows, and therefore $f(x) = 0$ and $f(y_i)>\epsilon$. This contradicts the continuity of $f$. Therefore $y\mapsto \varrho^{(M)}(y,x)$ is continuous at $x$.

From the inequality $|\varrho^{(M)}(x,y)-\varrho^{(M)}(x',y')| \leq \varrho^{(M)}(x,x')+\varrho^{(M)}(y,y')$ we infer that the $\rho^{(M)}$ is jointly continuous.

 Let $u := \varrho^{(M)}_A$ resp. $u := \varrho^{(M)}_A \wedge T$. For every $n \in \NN$ there is a countable set of points $y_i = y_i^n \in X$, $i \in \NN$, such that $\{ B_{1/n}(y_i) \}$ covers $X$. By the previous lemma, for every $i \in \NN$ there is a function $v_i = v_i^n \in M$ with $v_i (A) = 0$ and $v_i(y_i) \geq u(y_i) - 1/n$. Moreover, we know that $v_i(y) \leq u(y)$ for all $ y \in X$. From the definition of $B_r(x)$ and the remark above we can infer, that $v_i(y) \geq v_i(y_i) - 1/n$ for all $ y \in B_{1/n}(y_i)$ and $u(y) \leq u(y_i) + 1/n$ for all $ y \in B_{1/n}(y_i)$. In this way we get
 \[ u(y) \leq u(y_i) + 1/n \leq v_i(y_i) + 2/n \leq v_i(y) + 3/n \;\mbox{ for all }\; y \in B_{1/n}(y_i). \]
 Therefore $u_n := \sup v_i^n \in M$ and $u - 3/n \leq u_n \leq u$. (W.l.o.g. $u_n \leq T$ if needed.) This gives $u = \sup u_n \in M$.
\end{proof}

\begin{remark}
 We could set $M_0 := \{ f \in \calD_{loc}^* \cap C(X) \mid \mu^{(d)}(f) \leq m \}$ and would get $\varrho_0(x, y) := \sup \{ f(x) - f(y) \mid f \in \calD_{loc}^* \cap C(X), \mu^{(d)}(f) \leq m \}$, but, in general, this $M_0$ does not satisfy the fourth point in the definition, $\sup\{f_i\} \in M_0$, and, in general, $\varrho_0$ is not an intrinsic metric in our sense, as we shall discuss in Section \ref{sec:strongly_local}. Therefore the direct construction of cut-off functions done below seems not possible with this metric. This ``maximal'' intrinsic metric may be useful in some tasks, however, and we know that $\varrho_0 = \sup_\text{$M$ is set of i.m.f.} \varrho^{(M)}$.
\end{remark}

\section{The strongly local case}\label{sec:strongly_local}

In the previous sections we have introduced a concept of intrinsic metric for general Dirichlet forms. Of course, intrinsic metrics have been well studied in the case of strongly local forms. In this section we discuss the crucial difference between the local and the non-local case: For local Dirichlet forms (under an additional continuity asumption) there is a maximal intrinsic metric. For non-local Dirichlet forms there is in general not a maximal intrinsic metric.

\medskip

\begin{lemma}
 Let $\calE$ be local and let $\varrho_1$ and $\varrho_2$ be two intrinsic metrics. Then $\varrho_1 \vee \varrho_2$ is an intrinsic metric.
\end{lemma}

\begin{proof}
 This is a direct consequence of the truncation property of $\mu^{(c)}$.
\end{proof}

The following example shows, that this in general is not true for non-local Dirichlet forms.
\begin{example}
 Let $X = \{1, 2, 3\}$, $m$ be the counting measure on $X$ and $\calE(u, v) = 2(u(1) - u(2))^2 + 2(u(2) - u(3))^2$.
 \[ \varrho_1(x, y) := \begin{cases}
            1 & \text{if}\ x \neq y \text{ and } \{x, y\} \supset \{3\} \\
            0 & \text{otherwise}
           \end{cases} \]
 \[ \varrho_2(x, y) := \begin{cases}
            1 & \text{if}\ x \neq y \text{ and } \{x, y\} \supset \{1\} \\
            0 & \text{otherwise}
           \end{cases} \]
 $\varrho_1$ and $\varrho_2$ are both intrinsic metrics. However, one can see directly that $\varrho_1 \vee \varrho_2$ is not an intrinsic metric. Thus, there cannot be any intrinsic metric bigger than both $\varrho_1$ and $\varrho_2$ as otherwise $\varrho_1\vee \varrho_2$ had to be an intrinsic metric by
 Corollary \ref{cor:IM_order}.
\end{example}

On the other hand, in the strongly local case there is a maximal intrinsic metric (whenever a continuity condition is satisfied).

\begin{theorem}\label{maximal-intrinsic-metric}
 Let $\calE$ be local and assume that $\varrho(x, y) := \sup \{ u(x) - u(y) : u \in \calD_{loc} \cap C(X), \mu^{(c)}(u) \leq m \}$ is continuous. Then $\varrho$ is an intrinsic metric and any other intrinsic metric is (pointwise) smaller or equal to $\varrho$
\end{theorem}

\begin{proof}
 We first show that $\varrho$ is an intrinsic metric. Under the (somewhat stronger) assumption that $\varrho$ generates the original topology this is proven in \cite{Sturm-94b}, see the appendix of \cite{BoutetdeMonvelLS-08} for a somewhat alternative approach as well. An explicit proof, which is based on \cite{Sturm-94b} and assumes only continuity of $\varrho$,  can be found in \cite{Wingert-07}. For the convenience of the reader and as we have already gathered the necessary ingredients we include a sketch here.

 Choose $M := \{ u \in \calD_{loc} \cap C(X) \mid \mu^{(c)}(u) \leq m \}$. We prove for $f_i \in M$, $i \in \NN$ with $f_i \leq g$ for some function $g < \infty$ that we have $f := \sup f_i \in M$. First observe that $f$ is continuous, because $\varrho$ is, und thus $f \in L^\infty_{loc}$. By Lemma \ref{lem:3.5} we can assume $0 \leq f \leq M$. Define $u_n := \sup_{1 \leq i \leq n} f_i$. Then $f = u := \sup u_n$, $u_n$ is increasing and $\mu^{(c)}(u_n) \leq m$.

 Following the proof of Theorem \ref{thm:cool} we can now show that $u \in \calD_{loc}$ and
 \[ \mu^{(c)}(u) \leq m. \]
 The proof in fact simplifies because $J = 0$. We can then use the second part of Theorem \ref{thm:3.9} to conclude that $\varrho$ is an intrinsic metric.

 \smallskip

 The statement about maximality is clear (as for any intrinsic metric $\varrho'$ and each $x\in X$, the function $u(y) := \varrho' (x,y)$ can be plugged into the definition of $\varrho$).
\end{proof}

\section{Absolutely continuous jump measure}\label{Absolutely}

Here we take a look at a special situation that the Dirichlet form is purely non-local with a density. In this case, we can give a simple sufficient condition for a pseudo-metric to be intrinsic.

\medskip

We consider the following situation (S) described next:
\begin{itemize}
\item[(S)]
The function $j : X \times X \rightarrow [0,\infty]$ is measurable, symmetric (i.e., $j(x,y)=j(y,x)$ for all $x, y \in X$) and finite outside the diagonal,
 and the set $\tilde{\calD}$ of $u \in C_c(X)$ with
 \[ \int_{X \times X - d} (u(x) - u(y))^2 j(x, y) m(\rmd x) m(\rmd y) <\infty \]
 is dense in $C_c(X)$ with respect to the supremum norm.
\end{itemize}

 By standard Fatou-type arguments, the form
 \[ \calE(u) := \int_{X \times X - d} (u(x) - u(y))^2 j(x, y) \: m(\rmd x) m(\rmd y) \]
 defined for $u \in \tilde{\calD}$ is closable in $L^2(X,m)$. Let $\calD$ be the closure of $\tilde{\calD}$ with respect to $\calE_1$. Then, $(\calE, \calD)$ is a regular Dirichlet form on $L^2(X,m)$.

\begin{example}\label{ex:disclapl2}
The discrete Laplacian on $\ZZ^d$, discussed already in Example \ref{ex:disclapl}, falls into the situation (S) considered in this section. In that case, $j(x,y)=1$ if $|x-y|=1$ and $j(x,y)=0$ otherwise. In Subsection \ref{sec:graph} we will extend this example to general graphs.
\end{example}

\begin{example}\label{ex:fraclapl}
For $0<s<1$ we consider $\calE(u)=\|(-\Delta)^{s/2}u\|^2_{L^2(\RR^d)}$ for $u\in W^{s,2}(\RR^d)$, where $(-\Delta)^s$ is defined via the Fourier transform. By Plancharel's theorem we find that
\[ \calE(u) = a_{s,d} \iint_{\RR^d\times\RR^d} \frac{(u(x)-u(y))^2}{|x-y|^{d+2s}} \rmd x\rmd y \]
with $a_{s,d}^{-1} = 4 \int_{\RR^d} |z|^{-d-2s} \sin^2(z_d/2) \rmd z$. (The precise value of this constant is not important for us.) Hence we are again in situation (S) with $j(x,y) = a_{s,d} |x-y|^{-d-2s}$.
\end{example}

\smallskip

The next result is the converse to Lemma \ref{lem:4.8}.

\begin{theorem}\label{thm:converse}
 Assume the situation (S). Let $\varrho$ be a continuous pseudo-metric on $X$ such that
 \[ \int_{X - \{x\}} \varrho(x, y)^2 j(x, y) \: m(\rmd y) \leq 1 \]
 for almost all $x \in X$. Then $\varrho$ is an intrinsic metric.
\end{theorem}

\begin{proof}
 By
 \[ | \varrho_A(x) \wedge T - \varrho_A(y) \wedge T | \leq \varrho(x, y), \]
 we obtain
 \[ \sint{E \times X - d} (\varrho_A \wedge T (x) - \varrho_A \wedge T (y))^2 j(x, y) \: \rmd x \rmd y \leq \sint{E \times X - d} \varrho(x, y)^2 j(x, y) \: \rmd x \rmd y \leq \int_E \: \rmd x \]
 and hence $\rmd \mu^{(d)}(\varrho_A\wedge T) \leq \rmd x$.
Thus, it remains to show that $\varrho_A \wedge T \in \calD_{loc}$. For this we note that $\varrho_A \wedge T \wedge \varphi \in \widetilde{D} \subset \calD$ for any $\varphi \in \calD \cap C_c(X)$. This gives the desired statement.
\end{proof}

\section{Dirichlet forms with finite jump size}\label{Finite}

In this section we introduce the jump size of a Dirichlet form with respect to a given intrinsic metric. As can be expected, there is no big difference between Dirichlet forms with finite jump size and local Dirichlet forms. Dirichlet forms with infinite jump size, however, will be much more difficult to handle (see the later sections).\medskip

\begin{definition}
 Let $\varrho$ be an intrinsic metric. We define the \emph{jump size} of $\calE$ with respect to $\varrho$ by
 \[ \inf \{ t \geq 0 \mid J(\{ (x, y) \in X \times X - d \mid \varrho(x, y) > t\}) = 0\}. \]
\end{definition}

To illustrate this notion, let us consider the following

\begin{example}
For the discrete Laplacian on $\ZZ^d$ from Examples \ref{ex:disclapl} and \ref{ex:disclapl2} we know that
$j(x,y)=1$ if $|x-y|=1$ and $j(x,y)=0$ otherwise, and we can take $\varrho(x,y)=(1/\sqrt{2d})|x-y|$. Therefore, the jump size with respect to this intrinsic metric is $s=1/\sqrt{2d}$.
\end{example}

\begin{example}
Assume that $\rho$ is a translation-invariant (i.e., $\rho(x,y)=\sigma(x-y)$) intrinsic metric for the fractional Laplacian $(-\Delta)^s$, $0<s<1$, from Example \ref{ex:fraclapl}. If $\sigma$ is unbounded, then the jump size is infinite. (Indeed, if $\sigma$ is unbounded, then $\int_{\sigma(z)>t} |z|^{-d-2s} \rmd z>0$ for any $t>0$, and hence $\iint_{\sigma(x-y)>t} |x-y|^{-d-2s} \rmd x\rmd y = \infty$ for any $t>0$.) We refer to Subsection \ref{sec:fraclapl} for a discussion of intrinsic metrics for $(-\Delta)^s$.
\end{example}

\begin{remark}
 Note that the jump size is $0$ if $J \equiv 0$, i.e., if $\calE$ is local. The converse holds as well provided $\varrho$ separates points.
\end{remark}

For later use we note the following statement. Recall that the cut-off function $\eta_{E,a}$ was introduced in Definition \ref{cutoff}.

\begin{proposition}\label{hilfsaussage}
 Let $s$ be the jump size of $\calE$, let $E\subset X$ and $a>0$ be given and set $\eta:=\eta_{E,a}$.
 Then $a^2 \mu^{(d)}(\eta) \leq \One_{A_{s+a}(E)} m$.
\end{proposition}

\begin{proof}
 Let $\varphi \in C_c(X)$ and $\varphi \geq 0$. Note that $a \eta = (a - \rho_E)^+$. Hence,
 a direct calculation gives
 \begin{align*}
 a^2 \int_X \varphi \: \rmd \mu^{(d)}(\eta) & = a^2 \sint{X \times X - d} \varphi(x) (\eta(x) - \eta(y))^2 \: J(\rmd x, \rmd y) + a^2 \int_X \varphi \: \rmd \mu^{(c)}(\eta)\\
  &= \sint{A_{a+s}(E) \times X - d} \varphi(x) ((a - \varrho_E(x))^+ - (a - \varrho_E(y))^+)^2 \: J(\rmd x, \rmd y) \\
  & \quad + \sint{A_a(E)} \varphi \: \rmd \mu^{(c)}(\varrho_E) \\
  &\leq \sint{A_{a+s}(E) \times X - d} \varphi(x) (\varrho_E(x) - \varrho_E(y))^2 \: J(\rmd x, \rmd y) + \sint{A_{a+s}(E)} \varphi \: \rmd \mu^{(c)}(\varrho_E) \\
  &= \int_{A_{a+s} (E)} \varphi \rmd \mu^{(d)}(\varrho_E)\\
  &\leq \sint{A_{a+s}(E)} \varphi \: \rmd m.
 \end{align*}
 In the last step, we used that $\varrho$ is an intrinsic metric.
\end{proof}

\section{Measure perturbations and generalized eigenfunctions}\label{Measure}

We will be concerned with perturbations of Dirichlet forms by measures. Let
\[ \calM_0 := \{\nu : \calB \rightarrow [0, \infty] \mid \nu \text{ $\sigma$-additiv, } \nu \ll \capp \} \]
be the set of positiv measures which charge no set of capacity zero. For measures $\nu \in \calM_0$ we define
\[ \calD(\nu) := \{u \in \calD \mid \tilde{u} \in L^2(X, \nu)\} \]
and
\[ \nu(u, v) := \int_X \tilde{u} \tilde{v} \: \rmd \nu. \]
We emphasize that we shall use the notation $\nu(u)=\nu(u,u)$ in accordance with our convention about quadratic forms.

This makes sense as $\tilde{u}$ is defined quasi-everywhere. Let
\[ \calM_1 := \{\nu \in \calM_0 \mid \exists q < 1, C_q \geq 0 : \nu(u) \leq q \calE(u) + C_q \| u \|^2 \;\forall u \in \calD \}. \]
Then $\calE + \nu^+ - \nu^-$ is a closed form for $\nu^+ \in \calM_0$ and $\nu^- \in \calM_1$ and there is an associated selfadjoint operator $H$; see, e.g., \cite[Theorem VIII.16]{ReedS-72}.

Starting from this section we shall take a look at the perturbed form
\[ h := \calE + \nu := \calE + \nu^+ - \nu^-\]
with $\nu^+ \in \calM_0$ and $\nu^- \in \calM_1$.
The form domain of the perturbed form will be denoted by $\calD (h) = \calD \cap \calD (\nu_+)$. This space gives rise to a local space $\calD_{loc}^\ast (h)$ in the way discussed above.

Recall that in Subsection \ref{sec:locformdom} we have given sense to $\calE(u,\phi)$ for $u\in \calD_{loc}^*$ and $\varphi\in \calD$ with compact support. In a similar way, the expression $h(u, \varphi)$ is meaningful for $u\in \calD_{loc}^* (h)$ and $\varphi\in \calD (h)$ with compact support.

\begin{definition}
 A function $u \in \calD_{loc}^* (h) \setminus \{ 0 \}$ is called a \emph{generalized eigenfunction} corresponding to the \emph{generalized eigenvalue} $\lambda \in \RR$ if $h(u, \varphi) = \lambda(u, \varphi)$ for all $\varphi \in \calD (h)$ with compact support.
\end{definition}

\section{The ground state representation}\label{Ground}

This section deals with the ground state representation. This is an old topic, going back (at least) to Jacobi, and we refer to the papers cited in this paragraph for historical remarks. Recently, this representation has been investigated for strongly local Dirichlet forms in \cite{LenzSV}. In a purely non-local situation very similar to the one considered in this paper (and, in fact, even allowing for a non-linearity) it has been derived in \cite{FS}, extending previous special cases in \cite{FLS,FSW}. Here, we generalize this formula to our context with both a local and a non-local part and we provide a simple proof along the lines of \cite{LenzSV}. In fact, we give two version of the result with slightly different assumptions. One version can be thought of to work for the infimum of the spectrum and the other version works for any point in the spectrum.

\bigskip

The next theorem gives an effective bound on the infimum of the spectrum by representing the form. It requires that the generalized eigenfunction has a fixed sign.

\begin{theorem}
 Let $h = \calE + \nu$ with $\nu^+ \in \calM_0$, $\nu^- \in \calM_1$ and $\calE$ a regular Dirchlet form. Let $u$ be a generalized eigenfunction to the eigenvalue $\lambda$ with $u \neq 0$ q.e.\ and $u^{-1} \in \calD_{loc}^*$. Then the formula
 \[ h(\phi, \psi) - \lambda (\phi,\psi) = \sint{X \times X} u(x) u(y) \: \rmd \Gamma(\phi u^{-1}, \psi u^{-1}) \]
 holds true for all $\phi, \psi \in \calD (h) $ with $\phi u^{-1}, \psi u^{-1} \in \calD_{loc}^* (h)$ and $\phi \psi u^{-1} \in \calD_c (h) $. If $u^{-1} \in \calD_{loc}^* (h) \cap L_{loc}^\infty$ the formula holds true for all $\phi, \psi \in \calD (h) \cap L_c^\infty$.
\end{theorem}

Here, if $F$ is a space of functions on $X$, $F_c$ denotes the subset of elements in $F$ with compact support.

\begin{proof}
 We follow the argument given in \cite{LenzSV}. Without loss of generality we assume $\lambda = 0$ and $k = 0$. The Leibniz rule gives
\[ 0 = \Gamma(u, 1) = \Gamma(u, u u^{-1}) = u^{-1}(x) \Gamma (u,u) + u(y) \Gamma (u, u^{-1}). \]
 Using the fact that $u$ is a generalized eigenfunction, the Leibniz rule and the preceeding formula we can calculate
 \begin{align*}
 h(\phi,\psi) &= \calE (\phi,\psi) + \nu (\phi, \psi)\\
 &=  \calE(\phi, \psi) + \nu(\phi \psi u^{-1}, u) \\
 &= \calE(\phi, \psi) - \calE(\phi \psi u^{-1}, u) \\
 &= \calE(\phi, \psi) - \sint{X \times X} u(x) u(x)^{-1} \: \rmd \Gamma(\phi \psi u^{-1}, u) \\
 &= \calE(\phi, \psi) + \sint{X \times X} u(x) u(y) \: \rmd \Gamma(\phi \psi u^{-1}, u^{-1}) \\
 &= \calE(\phi, \psi) + \sint{X \times X} u(x) u(y) \phi(x) \: \rmd \Gamma(\psi u^{-1}, u^{-1}) + \sint{X \times X} u(x) \psi(y) \: \rmd \Gamma(\phi, u^{-1}) \\
 &= \sint{X \times X} u(x) u(y) \phi(x) \: \rmd \Gamma(u^{-1}, \psi u^{-1}) + \sint{X \times X} u(x) \: \rmd \Gamma(\phi, \psi u^{-1}) \\
 &= \sint{X \times X} u(x) u(y) \: \rmd \Gamma(\phi u^{-1}, \psi u^{-1}).
 \end{align*}
 This gives the first statement. The last statement then follows by Theorem \ref{thm:products}.
\end{proof}

 The argument given above can be modified to give the following results. There, we do not need the assumptions $u>0$ and $u^{-1} \in \calD_{loc}^*$ but then have stronger restrictions on $\phi$ and $\psi$.

\begin{theorem}
 Let $h = \calE + \nu$ with $\nu^+ \in \calM_0$ and $\nu^- \in \calM_1$. Let $u$ be a generalized eigenfunction to the eigenvalue $\lambda$. Then,
 \[ h(u \phi, u \psi) - \lambda (u \phi, u \psi) = \sint{X \times X} u(x) u(y) \: \rmd \Gamma(\phi, \psi) \]
 for all $\phi, \psi \in \calD (h) \cap L_c^\infty$ whenever $u \phi, u \psi, u \phi \psi \in \calD (h)$.
 In particular, the formula holds for all $\phi, \psi\in \calD (h) \cap L_c^\infty$ if  $u \in L_{loc}^\infty$.
\end{theorem}
\begin{proof} Without loss of generality we can assume $k=0$ and $\lambda =0$. Using the Leibniz rule repeatedly we calculate
\begin{align*}
 \calE(u \phi, u \psi) + \nu(u \phi, u \psi) &= \int_X \: \rmd \mu^{(d)}(u \phi, u \psi) + \nu(u, u \phi \psi) \\
 &= \int_X u \: \rmd \mu^{(d)}(\phi, u \psi) + \int_X \phi \: \rmd \mu^{(d)}(u, u \psi) + \nu(u, u \phi \psi) \\
 &= \sint{X \times X} u(x) u(y) \: \rmd \Gamma(\phi, \psi) + \sint{X \times X} u(x) \psi(x) \: \rmd \Gamma(\phi, u) \\
 &\qquad + \int_X \: \rmd \mu^{(d)}(u, u \phi \psi) - \int_X u \psi \: \rmd \mu^{(d)}(u, \phi) + \nu(u, u \phi \psi) \\
 &= \sint{X \times X} u(x) u(y) \: \rmd \Gamma(\phi, \psi) + \calE (u,u\phi\psi) + \nu(u, u \phi \psi) \\
 &= \sint{X \times X} u(x) u(y) \: \rmd \Gamma(\phi, \psi).
 \end{align*}
 In the last step we used that $u$ is a generalized eigenfunction. This finishes the proof.
\end{proof}

\begin{remark}
 Let us note that the right hand side in the formula has the sub-Markovian property if $u \geq 0$. This observation plays a crucial role in the proof of spectral estimates for the perturbed form $h$ in \cite{FLS}.
\end{remark}

An immediate consequence of the first theorem of this section is the following Allegretto-Piepenbrink-type result.

\begin{corollary}
Let $h = \calE + \nu$ with $\nu^+ \in \calM_0$ and $\nu^- \in \calM_1$ and $\calE$ a regular Dirchlet form. Let $u\geq 0 $ be a generalized eigenfunction to the eigenvalue $\lambda$ with $u^{-1} \in \calD_{loc}^* \cap L_{loc}^\infty$. Then $h\geq \lambda$.
\end{corollary}
\begin{proof} By the first theorem of this section we have $h(\phi)\geq \lambda \|\phi\|^2$ for all $\phi \in \calD \cap L_c^\infty$. As such $\phi$ are dense in the domain by regularity, the statement follows.
\end{proof}

\section{A Caccioppoli-type inequality}\label{sec:caccio}
In this section we will estimate the energy measure of generalized eigenfunctions. For strongly local Dirichlet forms a version can be found in \cite{BoutetdeMonvelLS-08}. Our discussion is similar to the discussion therein.

\bigskip

\begin{theorem}\label{caccio}
 Let $\calE$ be a regular Dirichlet form, $\nu^+ \in \calM_0$ and $\nu^- \in \calM_1$ and $q\in (0,1)$ with $\nu^-(u) \leq q \calE (u) + C_q \| u \|^2$ be given and set $h = \calE + \nu_+ - \nu_-$. Then, for any $\lambda \in \RR$, there  exists a constant $C = C(\lambda, \nu^-)$ with
 \[ \int_X \eta^2 \: \rmd \mu^{(d)}(u) \leq C(\lambda, \nu^-) \left( \| u \eta \|^2 + \int_X \tilde{u}^2 \: \rmd \mu^{(d)}(\eta) \right) \]
 for any $u \in \calD_{loc}^*$, $\eta \in \calD \cap C_c(X)$ with $\eta u, \eta^2 u \in \calD$ and $h(u, u \eta^2) \leq \lambda(u, u \eta^2)$.
\end{theorem}

\begin{proof}
 W.l.o.g. $k = 0$. The Leibniz rule from Subsection \ref{Leibnizrule} yields
 \[ \lambda \| u \eta \|^2 - \nu(u \eta) \geq \calE(u, u \eta^2)
 = \sint{X \times X} \eta^2(x) \: \rmd \Gamma(u) + \sint{X \times X} u(x) (\eta(x) + \eta(y)) \: \rmd \Gamma(u, \eta), \]
and, by assumption, we have
 \[ - \nu(u \eta) \leq q \calE(\eta u ) + C_q \| u \eta \|^2. \]
 Finally, Leibniz rule again shows
 \[ \calE(\eta u) = \sint{X \times X} \eta^2(x) \: \rmd \Gamma(u) + 2 \sint{X \times X} \tilde{u}(x) \eta(y) \: \rmd \Gamma(u, \eta) + \sint{X \times X} \tilde{u}(x)^2 \: \rmd \Gamma(\eta ). \]
 Let us now assume the last integral to be finite (otherwise the claim is still true). We now set
 \[ T:= (1-q) \sint{X \times X} \eta^2(x) \: \rmd \Gamma(u). \]
 Putting everything together we can estimate
 \begin{small}
 \begin{align*}
 T &\leq (\lambda + C_q) \| u \eta \|^2 + q \sint{X \times X} \tilde{u}(x)^2 \: \rmd \Gamma(\eta )
 + \sint{X \times X} \tilde{u}(x) (-\eta(x) + (2 q - 1) \eta(y)) \: \rmd \Gamma(u, \eta)\\
  &\leq (\lambda + C_q) \| u \eta \|^2 + (q+\frac{1}{4S}) \sint{X \times X} \tilde{u}(x)^2 \: \rmd \Gamma(\eta) + S \sint{X \times X} (-\eta(x) + (2 q - 1) \eta(y))^2 \: \rmd \Gamma(u ) \\
 &\leq (\lambda + C_q) \| u \eta \|^2 + (q+\frac{1}{4S}) \sint{X \times X} \tilde{u}(x)^2 \: \rmd \Gamma(\eta) + 4 S \max(q, 1-q)^2 \sint{X \times X} \eta(x)^2 \: \rmd \Gamma(u )
 \end{align*}
 \end{small}
 for all $S > 0$.
\end{proof}

The bound takes a simpler form if $\calE$ has finite jump size.

\begin{corollary} Assume the situation of the previous theorem.
 Let $u$ be a generalized eigenfunction for $h$ to the generalized eigenvalue $\lambda$. Assume further the jump size of $\calE$ to be $s < \infty$ and let $a > 0$ be such that $B_a(E)$ is relatively compact. Then
 \[ \int_E \: \rmd \mu^{(d)}(u) \leq C(\lambda, \nu^-) (1 + \frac{1}{a^2}) \| u \One_{B_{s+a}(E)} \|^2. \]
\end{corollary}

\begin{proof}
 Apply the previous theorem with $\eta = \eta_{E, a}$ from Definition \ref{cutoff} and use Proposition \ref{hilfsaussage}.
\end{proof}

\section{A Shnol'-type inequality}

We first state a Shnol'-type inequality for non-local Dirichlet forms with not necessarily finite jump size. This inequality, together with a Weyl-type characterization of the spectrum will be our main tool in the next section for characterizing of the spectrum in terms of generalized eigenfunctions. The proof we present here mimics the proof in \cite{BoutetdeMonvelLS-08} for local Dirichlet forms.

\bigskip

Recall that $\eta_{E, a}$ has been introduced in Definition \ref{cutoff} above. Moreover, recall that $\calE_1 (u) = \calE (u) + \|u\|^2$.

\begin{theorem}\label{thm:shnol_inf}
 Let $h = \calE + \nu$ with $\nu^+ \in \calM_0$ and $\nu^- \in \calM_1$, and let $u$ be a generalized eigenfunction of $h$ corresponding to a generalized eigenvalue $\lambda\in\RR$. Let $\varrho$ be an intrinsic metric and let $E \subset X$, $a, s > 0$ such that $B_{2a + s} (E)$ is relatively compact. Put $\eta_1 := \eta_{E, a}$ and $\eta_2 := \eta_{A_{a+s} (E), a}$. Then there is a constant $C = C (\lambda,a,s,\nu_-)$ such that for any $v \in \calD(h)$
 \[ \bigl| (h - \lambda) (u \eta_1^2, v) \bigr|^2 \leq C \calE_1(v ) \left( \int_X \tilde{u}^2 \: \rmd \mu^{(b)}(\eta_1 ) + \int_X \tilde{u}^2 \: \rmd \mu^{(b)}( \eta_2) + \| u \One_{A_{2a+s}(E)} \|^2 \right). \]
\end{theorem}

\begin{proof}
 For convenience we omit the tilde on $u$ and $v$. Using the fact that $u$ is a generalized eigenfunction and Leibniz rule, we compute
 \begin{align*}
 \bigl(h - \lambda \bigr)(u \eta_1^2, v)
 & = \bigl(h - \lambda \bigr)(u \eta_1^2, v) - \bigl(h - \lambda \bigr)(u, \eta_1^2v) \\
 & = \int_X \: \rmd \mu^{(d)}(\eta_1^2 u, v) - \int_X \: \rmd \mu^{(d)}(u,\eta_1^2 v) \\
 &= \int_X u \: \rmd \mu^{(d)}(\eta_1^2, v) - \int_X v \: \rmd \mu^{(d)}(u, \eta_1^2) \\
 &= \sint{X \times X} \eta_1(x) ( u(x) + u(y) ) \: \rmd \Gamma(v, \eta_1)
 - \sint{X \times X} \eta_1(x) ( v(x) + v(y) ) \: \rmd \Gamma(u, \eta_1).
 \end{align*}
 The first term can be estimated with CSI by
 \[ 2 \calE(v)^{1/2} \left( \int_X u^2 \: \rmd \mu^{(d)}(\eta_1) \right)^{1/2}. \]
 We split the integration in the second term into the two regions $\varrho(x, y) \leq s$ and $\varrho(x, y) > s$. In order to bound the first part we use CSI and Proposition \ref{cutoffprop} and obtain
 \begin{align*}
 \left| \sint{\rho(x,y)\leq s} \eta_1(x) ( v(x) + v(y) ) \: \rmd \Gamma(u, \eta_1) \right|
 & \leq 2 \left( \int_{A_{a+s}(E)} \: \rmd \mu^{(d)}(u) \right)^{1/2} \left( \int_X v^2 \: \rmd \mu^{(d)}(\eta_1) \right)^{1/2} \\
 & \leq \frac{2}{a} \left(\int_{A_{a+s}(E)} \: \rmd \mu^{(d)}(u)\right)^{1/2} \left(\int_X v^2 \: \rmd m \right)^{1/2}.
 \end{align*}
 This is controlled by the Caccioppoli inequality, Theorem \ref{caccio}. The remaining term is given by
 \[ R := \sint{\varrho(x, y) > s} \eta_1(x) ( v(x) + v(y) ) \: \rmd \Gamma(u, \eta_1). \]
Using Cauchy-Schwarz and sorting the terms we can then estimate
\begin{align*}
 R^2 &\leq 4 \sint{\varrho(x, y) > s} v(x)^2 \: J(\rmd x, \rmd y) \cdot \sint{\varrho(x, y) > s} \eta_1(x)^2 (u(x) - u(y))^2 \: \rmd \Gamma(\eta_1, \eta_1) \\
 &\leq \frac{16}{s^2} \sint{\varrho(x, y) > s} v(x)^2 \varrho(x, y)^2 \: J(\rmd x, \rmd y) \cdot \sint{\varrho(x, y) > s} u(x)^2 \: \rmd \Gamma(\eta_1, \eta_1) \\
 \intertext{and with Lemma \ref{lem:4.8}}
 &\leq \frac{16}{s^2} \| v \|^2 \int_X u^2 \: \rmd \mu^{(d)}(\eta_1 ).
 \qedhere
\end{align*}
As $\mu^{(c)}$ is local, its contribution to $\mu^{(d)}$ can be estimated by the $L^2$-norm. This gives the desired statement.
\end{proof}

The bound takes a simpler form if $\calE$ has finite jump size.

\begin{corollary}\label{shnolcor}
Assume that $\calE$ has finite jump size $s < \infty$. Let $\varrho$ an intrinsic metric, $E \subset X$, $a>0$ such that $B_{2a+s}(E)$ is relatively compact. Put $\eta := \eta_{E, a}$. Let $u$ be a generalized eigenfunction for $h$ to the generalized eigenvalue $\lambda$ and let $v \in \calD(h)$. Then
\[ \bigl| h - \lambda \bigr| (u \eta^2, v) \leq C(a, q) \calE_1(v)^{1/2} \| \One_{A_{2s+2a}(E)} u \|. \]
\end{corollary}
\begin{proof} This follows from the previous theorem and Proposition \ref{hilfsaussage}.
\end{proof}

\section{The spectrum and generalized eigenfunctions}

Using the Shnol'-type inequality we shall prove that under certain conditions a generalized eigenvalue will be in the spectrum. First, we recall a Weyl-type criterion for the spectrum of a self-adjoint, semibounded operator.

\begin{lemma}[\cite{Stollmann-01}]
 Let $h$ be a closed form, bounded from below by a constant $C$, and let $H$ be the associated self-adjoint operator. Then the following assertions are equivalent:
 \begin{enumerate}
 \item[(i)] $\lambda \in \sigma(H)$.
 \item[(ii)] There are $u_n \in \calD(h)$ with $\|u_n\| = 1$ and
 \[ \sup_{\substack{v \in \calD(h) \\ h_{C+1}(v, v) \leq 1}} \bigl|(h - \lambda) \bigr(u_n, v)| \rightarrow 0 \; (n \rightarrow \infty). \]
 \end{enumerate}
\end{lemma}

Combined with the Shnol' type results of the previous section (Theorem \ref{thm:shnol_inf} and Corollary \ref{shnolcor}), this lemma gives the following two results.

\begin{corollary}\label{cor:corozwei}
 If $u$ is a generalized eigenfunction for $h$ corresponding to $\lambda$, and if there exists an increasing sequence $E_n \subset X$ such that $B_{2a+s}(E_n)$ is relatively compact and
 \[ \frac{\int_X \tilde{u}^2 \: \rmd \mu^{(b)}(\eta_1) + \int_X \tilde{u}^2 \: \rmd \mu^{(b)}(\eta_2) + \| u \One_{A_{2a+s}(E_n)} \|^2}{\| \One_{E_n} u \|^2} \rightarrow 0 \]
 for some $a, s > 0$ and $\eta_1$, $\eta_2$ chosen as in Theorem \ref{thm:shnol_inf}, then $\lambda \in \sigma(H)$.
\end{corollary}

\begin{corollary}\label{cor:coroeins}
 Let the jump size of $\calE$ be $s < \infty$. If $u$ is a generalized eigenfunction of $h$ corresponding to $\lambda$, and if there exists an increasing sequence $E_n \subset X$ such that $B_{2a+s}(E_n)$ is relatively compact and $\frac{\| \One_{A_{2s+2a}(E_n)} u \|}{\| \One_{E_n} u \|} \rightarrow 0$ for some $a > 0$, then $\lambda \in \sigma(H)$.
\end{corollary}

The case of infinite jump size is somehow more difficult to handle than the case of finite jump size, because we do not know anything about $\mu^{(b)}$ for a general Dirichlet form. In the next section we will give examples for a 'good' and a 'bad' $\mu^{(b)}$. For the case of finite jump size we will  precise the result here.

We will assume the following condititions (C):
\begin{itemize}
 \item[(C)] There is an intrinsic metric with finite jump size $s$. Let $a > 0$, $E_1 \subset X$, $k := 2s+2a$, $E_{n+1} := B_k(E_n)$, $F_1 := E_1$, $F_{n+1} := E_{n+1} \setminus E_n$ and $A_k(E_n) \subset F_{n+1} \cup F_n$ such that $E_n$ is compact and
\[ m(F_n) e^{- \gamma n} \rightarrow 0 \]
for all $\gamma > 0$. We set
\[ w(x) := w_n := \begin{cases} (n \sqrt{m(F_n)})^{-1} &: m(F_n) \neq 0 \\ 0 &: m(F_n) = 0 \end{cases} \]
for $x \in F_n$ and $w(x) := 0$ for $x \notin E_\infty := \bigcup_n E_n$. Clearly $w \in L^2(X, m)$.
\end{itemize}

\begin{theorem}\label{thm:geneigspec}
 Assume (C). If $u$ is a generalized eigenfunction corresponding to $\lambda$ and $wu \in L^2(X, m)$, then $\lambda \in \sigma(H)$.
\end{theorem}
\begin{proof}
 Assume $\lambda \notin \sigma(H)$. Then by Corollary \ref{cor:coroeins} there is an $N$ and an $\varepsilon > 0$ such that
 \[ \frac{\| \One_{A_{2s+2a}(E_{n})} u \|}{\| \One_{E_{n}} u \|} > \varepsilon \]
 for alle $n > N$. Define $u(n) := \| \One_{F_n} u \|^2$. Then we know that $u(N+1) + u(N) > \varepsilon^2 \sum_{n=1}^N u(n)$ and $\sum_{n=1}^\infty w_n^2 u(n) = c < \infty$. But this contradicts $m(F_n) e^{- \gamma n} \rightarrow 0 \;\forall \gamma > 0$.
\end{proof}

For reverse results -- finding generalized eigenfunctions for values in the spectrum -- we cite \cite{BoutetdeMonvelS-03b}:
\begin{theorem}
 Let $H$ be a self-adjoint, semi-bounded operator with ultra-contrac\-tive semigroup on an $L^2$-Hilbert space and let $w \in L^2$, $w > 0$. Then for spectrally almost all $\lambda \in \sigma(H)$ there exists a function $u$ with $wu \in L^2$ and
 \[ (Hf, u) = \lambda (f, u) \]
 for all $f \in \calD(H)$ with $w^{-1} f \in L^2$ and $w^{-1} Hf \in L^2$.
\end{theorem}

\textbf{Problem.} It is not clear whether, in general, a function $u$ as in the previous theorem is a generalized eigenfunction in our sense.

\section{Applications and examples}

In this section we discuss various situations in which our results can be applied.

\medskip

\subsection{Strongly local forms}

This situation includes the Laplace operator on subsets of Euclidean space as well as Laplace Beltrami operators on Riemannian manifolds or quantum graphs. As it has been investigated, for instance, in \cite{BoutetdeMonvelS-03b, BoutetdeMonvelLS-08, LenzSV} we refrain from further discussion.

\smallskip

\subsection{Dirichlet forms on graphs}\label{sec:graph}

We consider an undirected graph $(X,E)$ with a countable set of vertices $X$ and a set of edges $E$. We assume that there are no multiple edges, i.e., any edge is uniquely characterized by two vertices. Then, vertices $x, y \in V$ are neighbors, written as $x \sim y$, whenever there is an edge connecting them. The degree of a vertex is $deg (x) := |\{y \in V : x \sim y\}|$. We do not assume that this quantity is finite. A path of length $n$ between $x, y \in V$ consists of vertices $x_0, x_1, \ldots, x_n$ with $x_0 = x$, $x_n = y$ such that $x_i$ and $x_{i+1}$ are neighbors for all $i=0, \ldots, n-1$. We assume that for any two vertices there is a path of finite length between them. The graph distance $d_g$ between two vertices is then defined to be the length of the smallest path between these vertices. The graph distance induces the discrete topology.

Let now $m$ be a measure on $X$ of full support, i.e., $m$ is a map from $V$ to $(0,\infty)$. Let $j: E\longrightarrow (0,\infty)$ given with
\[ m'(x) := \sum_{x\sim y} j(x, y) < \infty \]
for every $x\in V$. Extend $j$ by zero to $X\times X$ and define
\[ \calE_{comp}(u, v) := \sum_{X\times X} j(x, y) (u(x) - u(y)) (v(x) - v(y)) \]
for $u, v \in C_c(X)$ and define $\calE$ to be its closure. Then $\calE$ is a regular Dirichlet form. We note that this construction includes the discrete Laplacian on $\ZZ^d$ that we discussed in Example \ref{ex:disclapl}. For a further detailed study of various (spectral) aspects of this Dirichlet form we refer to \cite{HK,KL1,KL2}.

Clearly, we see that
\[ \rmd \mu^{(b)}(u)(x) = \sum_{y \in X} j(x, y) (u(x) - u(y))^2. \]
We set
\begin{align*}
 M_1 &:= \{u : X \rightarrow \RR \mid (u(x) - u(y))^2 \leq \min \left( 1, \frac{m(x)}{m'(x)} \right) \;\forall x \sim y \} \\
 M_2 &:= \{u : X \rightarrow \RR \mid (u(x) - u(y))^2 \leq \min \left( 1, \frac{m(x)}{j(x, y) \cdot deg(x)} \right) \;\forall x \sim y \}
\end{align*}
Note that for $x\sim y$ and $u\in M_1$ or $u\in M_2$ the estimate $|  u(x) - u(y)| \leq 1$ holds.
 It is not hard to see that $\mu^d (u) (x) \leq m(x)$ for any $x\in X$ and $u\in M_1$ or $u\in M_2$. Thus,
 $M_1$ and $M_2$ are sets of intrinsic metric functions and $\varrho_1 := \varrho_{M_1}$ and $\varrho_2 := \varrho_{M_2}$ are intrinsic metrics. We have
\begin{align*}
 \min \left( 1, \inf_{x \in V} \sqrt{\frac{m(x)}{m'(x)}} \right) d_g \leq \varrho_1 \leq d_g
 \intertext{and}
 \min \left( 1, \inf_{x \in V} \sqrt{\frac{m(x)}{j(x, y) \cdot d(x)}} \right) d_g \leq \varrho_2 \leq d_g.
\end{align*}
The jump size of $\varrho_{1,2}$ is at most $1$.

\smallskip

We note that the topology generated by these intrinsic metrics is not necessarily the discrete topology:
\begin{example}\label{ex:topo}
 A graph with $\varrho_1$ and $\varrho_2$ not generating the discrete topology: \\
 Let
 \[ V := \{a_1, a_2\} \cup \bigcup_{n \in \NN} \{b_n, c_n\} \qquad E := \bigcup_{\substack{n \in \NN \\ i \in \{1, 2\}}} \{\{a_i, b_n\}, \{b_n, c_n\}\} \]
 \[ m \equiv 1 \]
 \[ j(a_i, b_n) = \frac{1}{n^2} \qquad j(b_n, c_n) = n. \]
Thus, $a_1$ and $a_2$ are connected to each $b_n$ and $c_n$ is connected to $b_n$ only and
 $m' (b_n)= n + 2/n^2$, $j(a_i, b_n)= 1/n^2$. In particular, a function in $M_1$ can between $a_i$ and $b_n$ change its value by not more than $1/n$. As this holds for all $n$ we infer that  $\varrho_1(a_1, a_2) = 0$. As $deg (a_i)=\infty$ we also have $\varrho_2(a_1, b_1) = 0$.
\end{example}
\smallskip

Even if the topology generated by $\varrho_j$ is not the original topology, our main theorems about the spectrum applies now (see also \cite{Kirsch-07} 7.1 for a special case)

\begin{theorem}
 Let $H$ be the operator associated to $\calE + \nu$ defined as in Section~\ref{sec:caccio}. Assume that the corresponding semigroup is ultra-contractive. Let $w$ be defined as in Theorem \ref{thm:geneigspec} and the assumptions of that theorem be fullfilled. Then:
 \begin{enumerate}
 \item[(i)] For every generalized eigenvalue $\lambda$ with generalized eigenfunction $u$ with $wu \in L^2$ we have $\lambda \in \sigma (H)$.
 \item[(ii)] For spectrally almost every $\lambda \in \sigma (H)$ there is a generalized eigenfunction $u$ to $\lambda$ with $wu \in L^2$.
 \end{enumerate}
\end{theorem}

\begin{proof}
 Because of the easy situation, we only have to show, that
 \[ \sum_{y \in X} j(x, y) u(y)^2 < \infty \qquad \forall x \in X, \]
 i.e., that $u \in \calD_{loc}^*$. This follows easily in view of
 \[ \sum_{y \in X} w(y)^2 u(y)^2 < \infty \]
 and the boundedness of $j(x, y) / w(y)^2$.
\end{proof}

\subsection{Exponentially decaying jumping kernel}

Let $X = \RR^d$ and
\[ \calE'(u, v) := \sint{X \times X - d} (u(x) - u(y)) (v(x) - v(y)) j(x, y) \: \rmd x \rmd y \]
\[ \calD(\calE') := \{ u \in L^2(X) \mid \calE(u) < \infty \} \]
with some measurable, symmetric jump kernel $j \geq 0$. Then $\calE$ is a Dirichlet form.

Now assume that $j$ satisfies
\[ j(x, y) \leq \begin{cases}
         C |x - y|^{- d - \alpha} & : |x - y| \leq 1 \\
         C e^{- \beta |x - y|} & : |x - y| > 1
        \end{cases} \]
with $C, \alpha, \beta > 0$ and $\alpha < 2$.
Then $C_c^\infty(X) \subset \calD (\calE')$, and the closure of the restriction of $\calE'$ to $C_c^\infty (X)$ is a regular Dirichlet form. We denote this form by $\calE$.

A particular case of such a Dirichlet form is the regular Dirichlet form associated with the pseudo-relativistic Hamiltonian
\[ \sqrt{- \Delta + m^2} - m. \]

Indeed, in this case
\[ j(x, y) = C |x - y|^{- d - 1} \Psi(m^{1/2} |x - y|) \]

where $\Psi = \Psi(r)$ is an explicit Bessel function (given by $\Psi (r) = \int_0^\infty s^{(n-1)/2} e^{-s/4-r^2/s} \: \rmd s$. The claimed bound $j$ now follows from standard facts about Bessel functions (see e.g.\ \cite[Thm 7.12]{LL}).

Returning to the general case, we see from the upper bound on $j$ that
\[ \int_{\RR^d - \{x\}} |x - y|^2 j(x, y) \: \rmd y \leq c^{-2} \]
for some $c > 0$ small enough independent of $x$. Therefore by Theorem \ref{thm:converse}, $\varrho(x, y) := c \cdot |x - y|$ is an intrinsic metric.

Let $E \subset \RR^d$ and let $\delta(x)$ be the euclidean distance of $x$ to the boundary of $E$. Then for $x \in \RR^d$ with $\delta(x) > \frac{a}{c} + 1$ we calculate
\begin{align*}
 \int_{\RR^d} (\eta_{E, a}(x) &- \eta_{E, a}(y))^2 j(x, y) \: \rmd y \leq \sint{|x-y| \geq \delta(x) - \frac{a}{c}} j(x, y) \: \rmd y \leq c_1 \int_{\delta(x) - \frac{a}{c}}^\infty r^{d-1} e^{- \beta r} \: \rmd r \\
 &\leq c_2 (1 + (\delta(x) - \frac{a}{c})^{d-1}) e^{- \beta (\delta(x) - \frac{a}{c})} \leq c_3 (1 + \delta(x)^{d-1}) e^{- \beta \delta(x)}
\end{align*}
with appropriately chosen constants $c_1, c_2, c_3$. This, together with the simple global bound $\rmd \mu^{(d)}(\eta_{E, a}) \leq \frac{1}{a^2} \rmd x$, shows that
\[ \rmd \mu^{(d)}(\eta_{E, a}) \leq c_4 (1 + \delta(x)^{d-1}) e^{- \beta \delta(x)} \:\rmd x. \]

Let $h = \calE + \nu$ as before and let $u$ be a generalized eigenfunction for $h$ corresponding to $\lambda$. Moreover, consider the surface integral $u(r) := \int_{K(r)} u(x)^2 \rmd o$ with $K(r) := \{ x \mid |x| = r \}$ and assume that
\[ \liminf_{N\to \infty} \int_0^\infty u(r) (1 + |N-r|))^{d-1} e^{-\beta |N-r|} \: \rmd r \left( \int_0^N u(r) dr \right)^{-1} = 0\]
Then Corollary \ref{cor:corozwei}, together with the estimate for $\rmd \mu^{(d)}(\eta_{E, a})$, implies $\lambda \in \sigma (H)$.

\subsection{$\alpha$-stable processes}\label{sec:fraclapl}

Let $X = \RR^d$. The fractional Laplacian $-(-\Delta)^{\alpha/2}$ is the infinitesimal generator of an $\alpha$-stable process for $0 < \alpha < 2$. As discussed in Example \ref{ex:fraclapl}, the corresponding Dirichlet form is (up to a constant) defined by
\[ \calE(u, v) := \sint{X \times X - d} (u(x) - u(y)) (v(x) - v(y)) |x-y|^{-d-\alpha} \: \rmd x \rmd y \]
\[ \calD := \{ u \in L^2(X) \mid \calE(u) < \infty \} \]

Here $c \cdot |x-y|$ is only an intrinsic metric for $c = 0$. But $c \cdot f(|x-y|)$ is an intrinsic metric for $f(x) = x^\beta \wedge x$, $0 < \beta < \alpha/2$ and an appropriate $c=c_\beta > 0$, as can be seen
from Theorem \ref{thm:converse} and the bound
\[ \int_0^\infty f(x)^2 x^{- 1 - \alpha} \:\rmd x < \infty. \]
We fix an intrinsic metric of this form.

If $x \in E \subset X$, then
\begin{align*}
 \int_{\RR^d} (\eta_{E, a}(x) &- \eta_{E, a}(y))^2 |x-y|^{- d - \alpha} \: \rmd y \leq \int_{\delta(x)}^\infty r^{d-1} r^{- d - \alpha} \: \rmd r \\
 &\leq \int_{\delta(x)}^\infty r^{- \alpha - 1} \: \rmd r = \frac{1}{\alpha} \delta(x)^{- \alpha},
\end{align*}
and thus $\rmd \mu^{(d)}(\eta_{E, c}) \leq \alpha^{-1} \delta(x)^{-\alpha} \:\rmd x$. If $x \notin E$, then $\rmd \mu^{(d)}(\eta_{E, c}) \leq \alpha^{-1} (\delta(x) - 1)^{-\alpha} \:\rmd x$ for $\delta(x) > 1$ by the same calculation. Thus, if $u$ is a generalized eigenfunction for $h=\calE+\nu$ corresponding to $\lambda$ with
\[ \frac{\int_0^{N-1} u(r) (N-r)^{-\alpha} \: \rmd r + \int_{N+2}^\infty u(r) (r-N-1)^{-\alpha} \: \rmd r + \int_{N-1}^{N+2} u(r) \: \rmd r}{\int_0^N u(r) dr} \stackrel{N \rightarrow \infty}{\longrightarrow} 0, \]
where $u(r) := \int_{K(r)} u(x)^2 do$, then $\lambda \in \sigma (H)$ by Corollary \ref{cor:corozwei}.

\begin{appendix}

\section{Proof of Proposition \ref{prop:continuity}}

We start with a lemma which is certainly well-known and is, in fact, a slight variant of \cite[Thm. IV.1.16]{Kato-66}.
\begin{lemma}\label{lem:weak_convergence}
 Let $h: \calD(h) \times \calD(h) \rightarrow \CC$ be a symmetric, sesquilinear, positive definite, closed form in some Hilbert space $\calH$. Let $u_n \in \calD(h)$ be a sequence, which is $h_1$ bounded and weakly convergent to $u$. Then $u \in \calD(h)$ and $u_n \rightarrow u$ $h_1$-weakly.
\end{lemma}
\begin{proof}
 Because $u_n$ is $h_1$-bounded there is an $h_1$-weak convergent subsequence $v_n$. Let now $v$ be the $h_1$-weak limit of an $h_1$-weak converging subsequence $v_n$. Then for all $w \in \calD(H)$ we have
 \[ (v_n, (H + 1) w) = h_1(v_n, w) \rightarrow h_1(v, w) = (v, (H + 1) w). \]
 With $(v_n, (H + 1) w) \rightarrow (u, (H + 1) w)$ we conclude that $u = v$. As this holds for any $h_1$-weak converging subsequence, we infer the statement.
\end{proof}

Moreover, we recall the following Lemma (see \cite[Lemma 3.2.2.]{FukushimaOT-94}).
\begin{lemma} \label{lem:Fukushima-weak} Let $H$ be a Hilbert space with inner product $\langle \cdot,\cdot \rangle$. Let $(\cdot,\cdot)$ be a sequilinear form on $H$ with $0 \leq (\cdot,\cdot) \leq \langle \cdot,\cdot \rangle$. Then, $(u_n,v)\to (u,v)$ whenever $u_n\to u$ weakly in the sense of $\langle\cdot,\cdot\rangle$.
\end{lemma}

We now come to the proof of Proposition \ref{prop:continuity}.
\begin{proof}[Proof of Proposition \ref{prop:continuity}]
 As the $(u_n)$ converge weakly with respect to $\calE_1$, the sequence $(u_n)$ is bounded with respect to the energy norm. In particular, there exists a $C>0$ with $\calE(u_n, u_n) \leq C$ for all $n\in \NN$ and Cauchy Schwartz inequality yields:

 \[ |\int f d\mu^{(*)} (u_n, v) |^2 \leq \int |f|^2 d\mu^{(*)} (u_n,u_n) \int 1 d\mu^{(*)} (v) \leq \|f\|_\infty C \calE (v). \]
 It therefore suffices to show the desired convergence for $f$ from a dense set in $C_c(X)$. By regularity, such a set is given by the $C_c(X) \cap \calD$. By the same argument it suffices to show the convergence for $v$ from $\calD \cap L^\infty$ which is a dense set in $\calD$. Let now $f \in C_c (X)\cap \calD $ and $v \in \calD \cap L^\infty$ be given.

 By the assumptions and Theorem \ref{thm:algebraic_structure} $\calE_1(u_n v, u_n v)$ and $\calE_1(u_n f, u_n f)$ are bounded. We have also $L^2$-weak convergence of $u_n \rightarrow u$ by Lemma \ref{lem:Fukushima-weak}. From this $L^2$-weak convergence of $u_n v \rightarrow u v$ and $u_n f \rightarrow u f$ follows. Thus, and Lemma \ref{lem:weak_convergence} gives $\calE_1$-weak convergence, too. We thus have
 \[ (\sharp) \;\: u_n \to u,\;\: u_n v \rightarrow u v,\:\;\mbox{and}\;\: u_n f \rightarrow u f\;\:\mbox{ $\calE_1$-weakly}. \]

 As is well known and can also be seen by the Leibniz rule given above  we can now express the measure $\mu^{(*)}$ via the forms $\calE^{(*)}$ as
 \[ \int f \rmd \mu^{(a)}(u_n, v) = \calE^{(a)}(u_n, f v)\]
 for $* = a$ and by
 \[ 2 \int f \rmd \mu^{(*)}(u_n, v) = \calE^{(*)}(u_n f, v) + \calE^{(*)}(v f, u_n) - \calE^{(*)}(u_n v, f) \]
 for $* = b, c, d$. The stated convergence now follow from Lemma \ref{lem:Fukushima-weak} and $(\sharp)$.
\end{proof}

\end{appendix}

%\bibliographystyle{amsplain}
%\bibliography{Bibliographie}

\end{document}